\documentclass[]{scrartcl}
\usepackage{amsmath, amsthm, amssymb, amsfonts}
\usepackage{graphicx}
\usepackage{xcolor}
\usepackage{booktabs}
\usepackage{soul}

\usepackage{thmtools}
\usepackage{thm-restate}

\usepackage{cite}

\theoremstyle{plain}

\newtheorem{theorem}{Theorem}
\newtheorem{proposition}[theorem]{Proposition}
\newtheorem{lemma}[theorem]{Lemma}
\newtheorem{corollary}[theorem]{Corollary}
\newtheorem{claim}[theorem]{Claim}
\newtheorem{example}{Example}
\newtheorem{problem}{Problem}
\newtheorem{obs}{Observation}

\theoremstyle{remark}

\theoremstyle{definition}

\newcommand{\N}{{\mathbb N}}
\newcommand{\Prob}{\mathbb{P}}
\newcommand{\E}{\mathbb{E}}

\newcommand{\referee}[1]{\textcolor{red}{\textbf{[~Referee Comment:} #1 ~\textbf{]}}}

\begin{document}
	
\title{Creating Subgraphs in Semi-Random Hypergraph Games}
\author{
Natalie Behague\thanks{Mathematics Institute, University of Warwick, Coventry, UK; e-mail: \texttt{natalie.behague@warwick.ac.uk}},
Pawe\l{} Pra\l{}at\thanks{Department of Mathematics, Toronto Metropolitan University, Toronto, ON, Canada; e-mail: \texttt{pralat@torontomu.ca}},
Andrzej Ruci\'nski\thanks{Faculty of Mathematics and Computer Science, Adam Mickiewicz University, Pozna\'n, Poland; e-mail: \texttt{rucinski@amu.edu.pl}}
}

\maketitle

\begin{abstract}
The semi-random hypergraph process is a natural generalisation of the semi-random graph process, which can be thought of as a one player game. For fixed $r < s$, starting with an empty hypergraph on $n$ vertices, in each round a set of $r$ vertices $U$ is presented to the player independently and uniformly at random. The player then selects a set of $s-r$ vertices $V$ and adds the hyperedge $U \cup V$ to the $s$-uniform hypergraph. For a fixed (monotone) increasing graph property, the player's objective is to force the graph to satisfy this property with high probability in as few rounds as possible.

We focus on the case where the player's objective is to construct a subgraph isomorphic to an arbitrary, fixed hypergraph $H$. In the case $r=1$ the threshold for the number of rounds required was already known in terms of the degeneracy of $H$. In the case $2 \le r < s$, we give upper and lower bounds on this threshold for general $H$, and find further improved upper bounds for cliques in particular. We identify cases where the upper and lower bounds match. We also demonstrate that the lower bounds are not always tight by finding exact thresholds for various paths and cycles.
\end{abstract}

\section{Introduction\label{sec:intro}}
In this paper, we consider a hypergraph generalization of the \textbf{semi-random graph process} suggested by Peleg Michaeli (see \cite{beneliezer2020fast} and \cite[Acknowledgements]{beneliezer2019semirandom}) and studied recently in \cite{beneliezer2019semirandom,beneliezer2020fast,gao2022fully,gao2022perfect,frieze2022hamilton,gao2020hamilton,BMPR,ham_cycles_preprint,gamarnik2023cliques,molloy2023matchings} that can be viewed as a ``one player game''. Such a generalization was first proposed in~\cite{BMPR} and also studied in~\cite{molloy2023matchings}.

The semi-random process on hypergraphs, $(G^{(r,s)}_t)_t$, is defined as follows. Fix integers $r \ge 1$ to be the number of randomly selected vertices per step, and $s > r$ to be the uniformity of the hypergraph. The process starts from $G^{(r,s)}_0$, the empty hypergraph on the vertex set $[n]:=\{1,2,\dots,n\}$, where $n \ge s$ (throughout, we often suppress the dependence on $n$). In each step $t\geq 1$, a set $U_t$ of $r$ vertices is chosen uniformly at random from $[n]$. Then, the player replies by selecting a set of $s-r$ vertices $V_t$, and ultimately the edge $e_t:=U_t \cup V_t$ is added to $G^{(r,s)}_{t-1}$ to form $G^{(r,s)}_{t }$. In order for the process to be well defined, we allow parallel edges. For instance, they are necessary if an $r$-element set $U$ has been chosen more than $\binom{n-r}{s-r}$ times. 

Note that the resulting hypergraph is $s$-uniform, or shortly an \textbf{$s$-graph}. If $r=1$ and $s=2$, then this is the semi-random graph process. Further, if we allowed the degenerate case $r=s$ (that is, the player chooses $V_t=\emptyset$ for all $t$), then $G_t^{(r,r)}=(U_1,\dots,U_t)$ would be just a uniform random $r$-graph process with $t$ edges selected with repetitions.

To avoid ambiguity in using the notions of uniform hypergraph and uniform distributions, we will use the synonym \emph{equiprobable} for the latter.

Let us mention briefly some other variants of the semi-random process. In~\cite{macrury2022sharp}, sharp thresholds were studied for a more general class of processes that includes the semi-random process. In~\cite{burova2022semi}, a random spanning tree of $K_n$ is presented, and the player keeps one of the edges. In~\cite{gilboa2021semi}, vertices are presented by the process in a random permutation. In~\cite{Harjas}, the process presents $k$ random vertices, and to create an edge the player selects one of them, and freely chooses a second vertex.

The goal of the player is to build an $s$-graph $G_t^{(r,s)}$ satisfying a given monotone property $\mathcal{P}$ as quickly as possible. To make it more precise we define the notions of a strategy and a threshold.


A \textbf{strategy} $\mathcal{S}$ of the player consists, for each $n \ge s$, of a sequence of functions $(f_{t})_{t=1}^{\infty}$, where for each $t \in \N$, $V_t:=f_t(U_1,V_1,\ldots, U_{t-1},V_{t-1},U_t)\in\binom{[n]}{s-r}$. Thus, the player's response, $V_t$,   is fully determined by $U_1,V_1, \ldots ,U_{t-1},V_{t-1},U_t$, that is, by the history of the process up until step $t-1$, and by the random set $U_t$ chosen at step $t$. Given $t:=t(n)$, let $G_{t}^{(r,s)}[\mathcal{S}]$ be  the sequence of semi-random (multi)-$s$-graphs obtained by following  strategy $\mathcal{S}$ for $t$ rounds; we shorten $G_{t}^{(r,s)}[\mathcal{S}]$  to $G_{t}^{(r,s)}$ when clear.

\medskip

Throughout the paper we write $a_n\gg b_n$ if $b_n=o(a_n)$, and say that an event holds \emph{asymptotically almost surely} (\emph{a.a.s.}) if it holds with probability tending to one as $n \to \infty$.
 For  a monotonically increasing property $\mathcal{P}$ of $s$-graphs, we say that a function $\tau^{(r)}_{\mathcal{P}}(n)$ is \emph{a threshold for $\mathcal{P}$} if the following two conditions hold:
\begin{itemize}
\item[(a)] there exists a strategy $\mathcal{S}$ such that if  $t:=t(n)\gg \tau^{(r)}_{\mathcal{P}}(n)$, then a.a.s.\
$G^{(r,s)}_t \in \mathcal{P},\;$

\item[(b)] for every strategy~$\mathcal{S}$, if $t:=t(n)=o( \tau^{(r)}_{\mathcal{P}}(n))$, then a.a.s.\
$G^{(r,s)}_t \not\in \mathcal{P}$.
\end{itemize}
 Observe that $\tau^{(r)}_{\mathcal{P}}(n)\ge\tau^{(r-1)}_{\mathcal{P}}(n)$ for all $r\ge2$. Indeed, one can couple the two games by always including one of the $r$ random vertices chosen in the $G^{(r,s)}_t$ process among the $s-(r-1)$ vertices selected by the player in the $G^{(r-1,s)}_t$ process.

\medskip

In~\cite{molloy2023matchings} it was shown for any $s\ge2$ and $r\in\{1,2\}$ that for both, $\mathcal{P}$ being the property of having a perfect matching and $\mathcal{P}$ being the property of having a loose Hamilton cycle, $\tau^{(r)}_{\mathcal{P}}(n)=n$ (in fact, the results are even sharper).

In this paper we focus on the problem of constructing a sub-$s$-graph of $G^{(r,s)}_t$ isomorphic to an arbitrary, \emph{fixed} $s$-graph $H$.
Let $\mathcal{P}_H$ be the property that $H \subseteq G^{(r,s)}_t$. We abbreviate $\tau^{(r)}_{\mathcal{P}_H}(n)$ to $\tau^{(r)}(H,n)$ and often suppress the dependence on $n$, writing simply~$\tau^{(r)}(H)$.

\medskip

It was proved by the authors and T. Marbach in~\cite{BMPR} that for $r=1$, that is, when just a single vertex is selected randomly at each step, the threshold $\tau^{(1)}(H)$ can be determined fully in terms of the degeneracy of $H$.
For a given $d \in \N$, a hypergraph $H$ is  \textbf{$d$-degenerate} if every sub-hypergraph $H'\subseteq H$ has minimum degree $\delta(H')\le d$. The \textbf{degeneracy} $d(H)$ of $H$ is the smallest value of $d$ for which $H$ is $d$-degenerate. Equivalently, $d(H)=\max_{H'\subseteq H}\delta(H')$, where $\delta(H)$ is the minimum vertex degree of a hypergraph $H$.

\begin{theorem}[Behague, Marbach, Pra{\l}at, Ruci\'{n}ski \cite{BMPR}]\label{thm:r=1}
	Let $s\ge2$ and $H$ be a fixed $s$-uniform hypergraph of degeneracy $d \in \N$. Then, $\tau^{(1)}(H) = n^{1-1/d}$.
\end{theorem}

Note that, in particular, for $s=2$ and any tree $T$ we have $d(T)=1$, and so $\tau^{(1)}(T)=1$. In fact, in this case one can easily show a stronger statement: there exists a strategy $\mathcal{S}$ such that a.a.s.\ $T \subseteq G^{(1,2)}_t[\mathcal{S}]$ for $t=|E(T)|$, as a.a.s.\ the first $t$ random vertices $u_1,\dots,u_t$, selected in the semi-random process $G^{(1,2)}_t$, are all distinct from each other, as well as, from a fixed vertex $u_0\in[n]$. On the other hand, for any (graph) cycle $C$, $d(C)=2$, yielding $\tau^{(1)}(C)=\sqrt n$ by Theorem~\ref{thm:r=1}.

A similar contrast takes place for $s>2$. A \emph{tight cycle} $C_m^{(s)}$ is an $s$-graph  with $m$ vertices and $m$ edges,  whose vertices can be ordered cyclically so that the edges are formed by the consecutive $s$-element segments in this ordering.
(E.g., the set of triples
 $123, 234, 345,456,567,671,712$ forms a copy of $C_7^{(3)}$ on $[7]$.)
 A \emph{tight path} $P_m^{(s)}$ is an $s$-graph  with $k=m+s-1$ vertices and $m$ edges,  whose vertices can be ordered linearly so that the edges are formed by the consecutive $s$-element segments in this ordering. Alternatively, it can be obtained from $C_{m+s-1}^{(s)}$ by removing $s-1$ consecutive edges, while keeping all vertices intact. (E.g., the set of triples
 $123, 234, 345,456,567$ forms a copy of $P_5^{(3)}$ on $[7]$.)
We have $d(P_m^{(s)})=1$, so $\tau^{(1)}(P_m^{(s)})=1$, while  $d(C_m^{(s)})=s$ and so $\tau^{(1)}(C_m^{(s)})=n^{1-1/s}$ (see Appendix, Claim~\ref{miu}).

\medskip

\section{New results}\label{or}

Our understanding of semi-hypergraph processes with $r \ge 2$ is far from complete. For property $\mathcal{P}_H$,  we can only prove a general lower  bound, show its optimality for certain classes of hypergraphs and  its suboptimality for others. We defer the proofs of these results to later sections. Throughout, for a hypergraph $H$, we will be using notation $v_H=|V(H)|$ and $e_H=|E(H)|$.

\subsection{Lower bound} Our general lower bound on $\tau^{(r)}(H)$, proved in Section~\ref{s3}, depends only on the number of vertices and edges of $H$ so, in a sense, it is also quite generic. Surprisingly, it provides the right answer for a broad class of $s$-graphs.

\begin{restatable}{theorem}{LBgeneral}
\label{thm:lower_bound_general}
	Let $k \ge s \ge r \ge 1$, and let $H$ be an $s$-graph with $k$ vertices and $m$ edges. Then, for every strategy $\mathcal{S}$, if  $t =o\left( n^{r-(k-s+r)/m}\right)$, then a.a.s.\ $G^{(r,s)}_t \not\in \mathcal{P}_H.$
 It follows that
 $$\tau^{(r)}(H) \ge n^{r-(k-s+r)/m}.$$
\end{restatable}

\begin{example}\label{Ex1}\rmfamily
Let $H$ be a 3-graph consisting of 5 vertices $a,b,c,d,e$ and 5 edges made by all triples from $\{a,b,c,d\}$ plus $\{c,d,e\}$ (see Figure~\ref{fig:example1}). Then, with $r=2$, $s=3$, and $k=m=5$,
we get $\tau^{(2)}(H) \ge n^{6/5}$.
\end{example}

\begin{figure}[h]
	\centering
	\includegraphics[scale=.6]{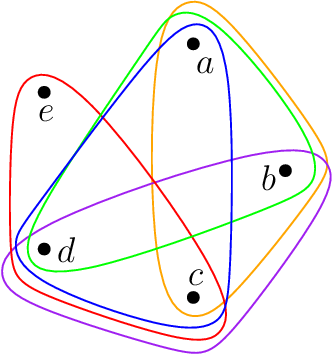}
	\caption{The $3$-graph $H$ described in Example~\ref{Ex1}.}
	\label{fig:example1}
\end{figure}

Of course, it might happen that the main ``bottleneck'' is not the original hypergraph $H$ as a whole, but one of its sub-hypergraphs $H' \subseteq H$. Trivially, creating  $H$ requires creating $H'$ in the first place, so we immediately obtain the following corollary. For $2\le r\le s$ and an $s$-graph $H$ with at least $s$ vertices, define
	$$f^{(r)}(H)=\frac {e_H}{v_H-s+r}\quad\mbox{and}\quad\mu^{(r)}(H)=\max_{H' \subseteq H,\ v_{H'}\ge s}f^{(r)}(H').$$
\begin{corollary}\label{cor:lower_bound}
	Let $1\le r\le s$ and $H$ be an  $s$-graph with at least $s$ vertices. Then,
	$$
	\tau^{(r)}(H) \ge n^{ r- 1/\mu^{(r)}(H)}.
	$$
\end{corollary}
If $\mu^{(r)}(H)=f^{(r)}(H)$, then we call such an $H$ \textbf{$r$-balanced}. This is, for instance,  the case for the tight cycle $C_m^{s}$ (see Appendix, Claim~\ref{miu}). For $r$-balanced $H$, the bounds in Theorem~\ref{thm:lower_bound_general} and Corollary~\ref{cor:lower_bound} coincide. However, for non-$r$-balanced $H$ Corollary~\ref{cor:lower_bound} may give a significantly better lower bound on $\tau^{(r)}(H)$.

\begin{example}\label{Ex2}\rmfamily
Let $H$ be as in Example \ref{Ex1} (see Figure~\ref{fig:example1}) and $H'$ be the clique $K_4^{(3)}$ on vertices $a,b,c,d$. Then, $H' \subseteq H$ and $f^{(2)}(H')=4/3>5/4=f^{(2)}(H)$.  It is easy to see that $\mu^{(2)}(H)=4/3$ and, in fact, $\tau^{(2)}(H) \ge n^{2-3/4}=n^{5/4}$.
\end{example}

\noindent Note that, by considering a single edge as $H'$, we always have $\mu_H^{(r)}\ge 1/r$, and so the lower bound in Corollary \ref{cor:lower_bound} cannot be less than 1. However, for most hypergraphs $H$ we have $\mu^{(r)}(H) > 1/r$, giving us a nontrivial lower bound on $\tau^{(r)}(H)$.
Note also that in the special case when $r=s\ge2$, we are looking at the random $s$-graph (with repeated edges) and the bound in Corollary \ref{cor:lower_bound} corresponds to the  threshold for appearance of a copy of a given $s$-graph $H$ (see~\cite[Chapter 3]{JLR} for the case $r=s=2$ or~\cite{dewar2017subhypergraphs} for non-uniform hypergraphs, though neither model allows edge repetitions, as we do here).

Finally, let us mention that we included  the case $r=1$, already covered by Theorem~\ref{thm:r=1}, to emphasize the potential weakness  of this general lower bound. Indeed, note that $\mu^{(1)}(H)> d(H)$ unless $\mu^{(1)}(H)=d(H)=1$, so in most cases the bound in Corollary~\ref{cor:lower_bound} is weaker than the optimal bound in Theorem \ref{thm:r=1}. However, for some $s$-graphs it is optimal. For example,  when $H$ is the tight path $P_m^{(s)}$, we have $\mu^{(1)}(H)=d(H)=1$ and Theorem~\ref{thm:r=1} yields $\tau^{(1)}(H)=1$ (see the comment after Theorem~\ref{thm:r=1} above).

\subsection{Upper bounds which match lower bound}\label{match}

We now identify a class of $s$-graphs $H$ for which we are able to establish an upper bound on $\tau^{(r)}(H)$ which matches the lower bound in Theorem~\ref{thm:lower_bound_general}.

For integers $1\le c< s\le k$, a $k$-vertex $s$-graph $S$ is called a \textbf{$c$-star} if each of its edges contains a fixed vertex set $C$ of size $|C|=c$. The set $C$ is then called the \textbf{center} of the star and the $(s-c)$-graph $S_1:=\{e\setminus C:\;  e\in S\}$ is the \textbf{flower} of the star $S$. A $c$-star is \textbf{full} when it has all $\binom{k-c}{s-c}$ edges, that is, if its flower is the complete $(s-c)$-graph $K_{k-c}^{(s-c)}$. A full star will be denoted by $S_k^{(s,c)}$.

An \textbf{$(s,c)$-starplus with $\lambda_1$ rays and excess $\lambda_2$} is defined as an $s$-graph obtained from a $c$-star $S$ with $\lambda_1$ edges by arbitrarily adding to it $\lambda_2$ edges not containing $C$ (but not adding any new vertices). For $c>1$, the additional edges may intersect the center~$C$ (but not contain it).
 Call the $(s-c)$-subgraph $H_1:=\{e\setminus C:\; C\subseteq e\in H\}$,  the \textbf{flower} of the starplus $H$, and the $s$-graph $H_2$ consisting of the $\lambda_2$ excess edges of $H$ -- the \textbf{cap} of $H$. (Note that  $H_1=S_1$, the flower of $S$.)

\begin{example}\label{starplus}\rmfamily
Let $V(S)=\{a,b,c,d,e,f\}$ and $E(S)$ consist of all 4-tuples containing $C:=\{a,b\}$ and one pair from $\{c,d,e,f\}$ except $\{e,f\}$ (see Figure~\ref{fig:example3}). Then $S$ is a $2$-star, though not full (as the edge $\{a,b,e,f\}$ is missing and thus its flower $S_1=K_4-\{e,f\}$). By adding to $S$ three edges: $\{c,d,e,f\}, \{a,c,d,e\}$, and $\{b,c,e,f\}$, we obtain a $(4,2)$-starplus $H$ with 5 rays and  excess 3. Its flower is the graph $H_1=S_1=K_4-\{e,f\}$ on vertex set $\{c,d,e,f\}$, while its cap $H_2$ consists of the three 4-tuples we have added to $S$.
\end{example}

\begin{figure}[h]
	\centering
	\includegraphics[scale=.6]{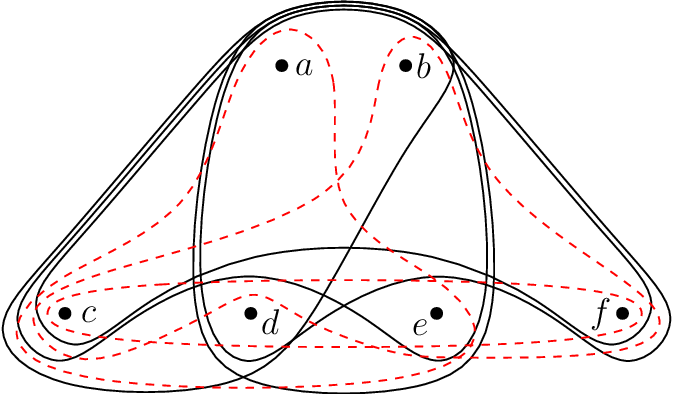}
	\caption{The $(4,2)$-starplus $H$ described in Example~\ref{starplus}, where the black solid edges are the edges of the $2$-star $S$ and the red dashed edges are the excess edges.}
	\label{fig:example3}
\end{figure}

For an $r$-graph $F$, $r\ge2$, let us define its \textbf{density}  $g(F)$ as $1/r$ if $e_F=1$, and  $\frac{e_F-1}{v_F-r}$ if $e_F>1$.  We call $F$ \textbf{edge-balanced} if all sub-$r$-graphs  $F'\subset F$ with $e_{F'}>0$ satisfy $g(F')\le g(F)$.
For starpluses  which are not too dense and whose flowers are edge-balanced, we can prove (see Section~\ref{s4}) the following.

 \begin{restatable}{theorem}{UBbalanced}
 \label{thm:upper_bound_balanced} For $r\ge2$ and $s>r$,
	let $H$ be an $(s,s-r)$-starplus on $k$ vertices with $\lambda_1$ rays and excess $\lambda_2$, such that
	\begin{equation}\label{eqn:ell-gen}
	\frac{\lambda_1+\lambda_2}{\lambda_1-1}\le\frac{k-s+r}{k-s},
	\end{equation}
and whose flower $H_1$ is edge-balanced.
	Then, there exists a strategy $\mathcal{S}$ such that, if
$t \gg n^{r-\frac{k-s+r}{\lambda_1+\lambda_2}}, $
then a.a.s.\ $G_t^{(r,s)} \in \mathcal{P}_H$.
	 Thus, combined with Theorem~\ref{thm:lower_bound_general},
	$$\tau^{(r)}(H)= n^{r-\frac{k-s+r}{\lambda_1+\lambda_2}} .$$
\end{restatable}

 It follows  that, in view of Corollary \ref{cor:lower_bound}, any $(s,s-r)$-starplus $H$ satisfying the assumptions of Theorem \ref{thm:upper_bound_balanced} is  $r$-balanced, a fact whose direct proof would be quite tedious (see Appendix, Proposition~\ref{starbal}, for a proof in the special case of full $(s,s-r)$-starplus defined prior to Corollary~\ref{thm:upper_bound_complete} below).

Note also that the assumptions of Theorem~\ref{thm:upper_bound_balanced} do  not impose any structural restrictions on the cap $H_2$. Therefore, once the flower $H_1$ of an $(s-r)$-star $S$ is edge-balanced and the parameters satisfy \eqref{eqn:ell-gen}, we can take any $s$-graph with $\lambda_2$ edges and $k$ vertices (edge-disjoint from $S$) as a cap, obtaining a whole family of  $(s,s-r)$-starpluses to which Theorem~\ref{thm:upper_bound_balanced} applies.

\medskip

\begin{example}\label{tc}\rmfamily
  One such class of  starpluses is defined in terms of  tight cycles.
 Every tight cycle is edge-balanced (see Appendix, Claim \ref{edge-bal}).
 Thus, every $(s,s-r)$-starplus $H$ on $k$ vertices whose flower is $H_1=C_{k-s+r}^{(r)}$ and whose cap $H_2$ has no more than $\frac{r-1}{k-s}(k-s+r)$ edges, satisfies the assumptions of Theorem \ref{thm:upper_bound_balanced}. In particular, for
\emph{the wheel} $H=W_k^{(s,s-r)}$  defined as an $(s,s-r)$-starplus $H$ with  $H_1=C_{k-s+r}^{(r)}$ and $H_2=C_{k-s+r}^{(s)}$, and for $k\le s+r-1$,  we have $\tau^{(r)}(H)= n^{r-1/2}$  (see Figure~\ref{fig:8wheel} for the wheel $W_8^{(5,1)}$ which satisfies the above assumptions with $s=5$ and $r=4$).

\begin{figure}
    \centering
    \includegraphics[scale=1]{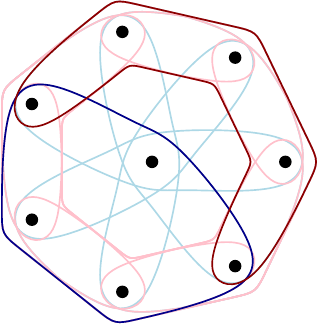}
    \caption{The wheel $W_8^{(5,1)}$, with the `ray' edges containing the centre in blue and the `excess' edges of the cap in red.}
    \label{fig:8wheel}
\end{figure}

\end{example}

\smallskip

Assumption \eqref{eqn:ell-gen} is quite restrictive, because its left-hand-side cannot be too large. It becomes more relaxed, though, when we enlarge $\lambda_1$. At the extreme, $\lambda_1$ can be as large as $\binom{k-s+r}r$. This leads to the following notion.

For integers $1\le c\le s\le k$, an $(s,c)$-starplus  $H$ with $\lambda_1$ rays and surplus $\lambda_2$ is called a \textbf{full $(s,c)$-starplus with excess $\lambda$}  if $H_1=K^{(s-c)}_{k-c}$, that is, the flower is a complete $(s-c)$-graph (and so $\lambda_1=\binom{k-c}r$), while $\lambda_2=\lambda$.  Alternatively, an $(s,c)$-starplus with excess $\lambda$ is an edge-disjoint  union of a full $c$-star $S_k^{(s,c)}$ and an $s$-graph $H_2$ with $\lambda$ edges and the same vertex set as the star. Note that, in this case, $H_1$ is edge-balanced (see Appendix, Claim~\ref{clicbal}) and thus, Theorem \ref{thm:upper_bound_balanced} immediately implies the following result.

\begin{corollary}\label{thm:upper_bound_complete} Let $r\ge2$ and $s>r$, and
	let $H$ be a full $(s,s-r)$-starplus on $k$ vertices with excess
	\begin{equation}\label{eqn:ell}
	\lambda\le\frac{r\binom{k-s+r}r-(k-s+r)}{k-s}.
	\end{equation}
	Then,
	$$\tau^{(r)}(H)= n^{r-\frac{k-s+r}{\binom{k-s+r}r+\lambda}} .$$
\end{corollary}

\begin{example}\label{e4}\rmfamily For $r=2$ and $s=3$ the upper bound on excess in \eqref{eqn:ell} is $k-1$, so Corollary~\ref{thm:upper_bound_complete} applies in this case to all $s$-graphs whose cap has the same number of edges and vertices. One example is the 3-uniform clique $K_5^{(3)}$ on 5 vertices which can be viewed as a $(3,1)$-starplus consisting of the full 1-star and  the cap forming a copy of $K_4^{(3)}$; so, $\tau^{(2)}(K_5^{(3)})= n^{2-\frac{4}{10}}=n^{8/5}$.

  Another  example, this time for $k=8$, is presented in Figure~\ref{fig:starplus} in Section~\ref{s4}. Here $H$ is the full $3$-uniform 1-star on 8 vertices topped  with the Fano plane; thus, $\tau^{(2)}(H)= n^{7/4}.$ Our last example is the full $(3,1)$-starplus on $k$ vertices whose cap $H_2$ is a tight cycle $C_{k-1}^{(3)}$; then,
$\tau^{(r)}(H)=n^{ 2-\tfrac{2}{k}}.$
 \end{example}

As hinted at in  Example~\ref{e4}, Corollary~\ref{thm:upper_bound_complete} can be sometimes applied to  complete $s$-graphs $K_k^{(s)}$.
Indeed, since cliques can be viewed as $(s,s-r)$-starpluses with excess  $\lambda=\binom ks-\binom{k-s+r}r$, assumption \eqref{eqn:ell}, for cliques, becomes
\begin{equation}\label{clique_starplus}
\binom{k}{s} \le \frac{k-s+r}{k-s}\left(\binom{k-s+r}{r} - 1 \right).
\end{equation}
Thus, in particular, Corollary \ref{thm:upper_bound_complete} covers cliques $K_{s+1}^{(s)}$, whenever $s\le r^2+r-1$, and cliques $K_{s+2}^{(s)}$, whenever $(s+2)(s+1)\le \tfrac12r(r+2)(r+3)$. For $r=2$ this covers the cliques $K^{(3)}_4, K^{(3)}_5, K^{(4)}_5, K^{(5)}_6$.
At the other extreme, when $r=s-1$, assumption \eqref{eqn:ell} for cliques becomes
\begin{equation}\label{in1}
\binom{k-1}{s} \le \frac{(s-1) \binom{k-1}{s-1}-(k-1)}{k - s},
 \end{equation}
 which holds whenever $k\le 2s-1$ (see Appendix, Claim~\ref{ineq}). So, in addition, Corollary~\ref{thm:upper_bound_complete} covers cliques   $K_7^{(4)}$ (for $r=3$),  $K_8^{(5)}$, $K_9^{(5)}$ (for $r=4$), and so on. The smallest case among cliques, not covered by Corollary~\ref{thm:upper_bound_complete},  is thus $K_6^{(3)}$ and $r=2$.\footnote{Recently, it was determined in~\cite{ShD} that $\tau^{(2)}(K_6^{(3)})
 =n^{16/9}$. }

\subsection{Upper bounds for general $s$-graphs and cliques}\label{2.3}

Corollary~\ref{thm:upper_bound_complete} can be used as a black box to derive a generic upper bound on $\tau^{(r)}(H)$ for any $H$, just in terms of its maximum degree and the number of edges. For $1\le d\le s$, let $\Delta_d(H)$ denote the maximum degree of a $d$-set of vertices of $H$, that is, the maximum number of edges that contain a given subset $D\subset V(G)$, $|D|=d$.

The following consequence of Corollary~\ref{thm:upper_bound_complete} has a very simple proof which, therefore, we present right after the statement.
Observe that the right-hand-side of~\eqref{eqn:ell} is an increasing function of $k$ (as $k=s$ is a root of the numerator viewed as a polynomial in $k$).

\begin{corollary}\label{cor:upper_bound} Let $s>r\ge2$ and let $H$ be an arbitrary $s$-graph. Further, let $k\ge v_H$ be the smallest integer for which \eqref{eqn:ell} holds with $\lambda:=e_H-\Delta_{s-r}(H)$. Then
	$$\tau^{(r)}(H)\le n^{r-\frac{k-s+r}{\binom{k-s+r}r+\lambda}} .$$
\end{corollary}
\proof Let $C\subset V(H)$, $|C|=s-r$, be a subset which achieves the maximum in the definition of $\Delta_{s-r}(H)$. Further, let $H'$ be the sub-$s$-graph of $H$ obtained by deleting all edges containing $C$ and let $\widehat H$ be the full $(s,s-r)$-starplus  on a $k$-vertex set containing $C$ and with the $\lambda$ surplus edges forming a copy of $H'$. (Alternatively, $\widehat H$ is obtained from $H$ by adding $k-v_H$ new vertices and $\binom{k-s+r}r-\Delta_{s-r}(H)$ new edges containing $C$.) See Figure~\ref{fig:cor_upper_bound} for an example.
As $\widehat H$ has excess $\lambda:=e_H-\Delta_{s-r}(H)$ satisfying \eqref{eqn:ell}, we may apply Corollary \ref{thm:upper_bound_complete} to it, obtaining the bound
$$\tau^{(r)}(\widehat H)\le n^{r-\frac{k-s+r}{\binom{k-s+r}r+\lambda}} .$$
 Clearly, $\widehat H\supseteq H$ and the statement follows by monotonicity.
 \begin{figure}
     \centering
     \includegraphics[width=.9\linewidth]{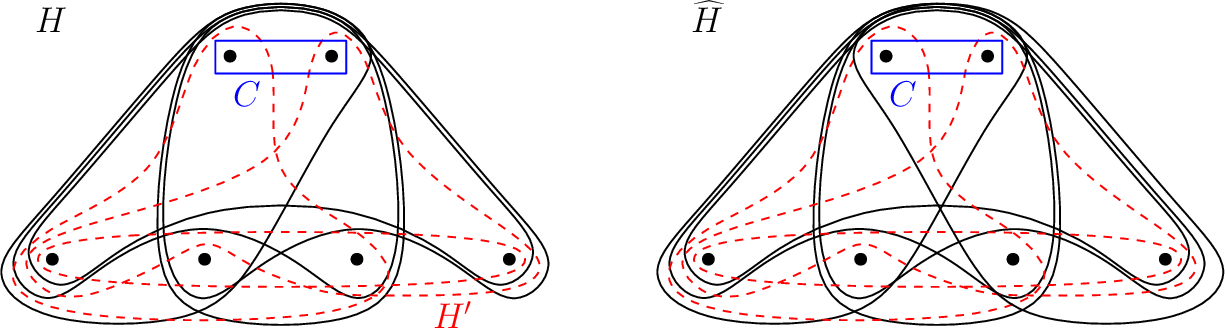}
     \caption{An example of the graphs $H'$ and $\widehat{H}$ constructed in the proof of Corollary~\ref{cor:upper_bound} for a given $H$. The red dashed edges represent the edges of $H'$.}
     \label{fig:cor_upper_bound}
 \end{figure}
 \qed

\medskip

The bound in Corollary~\ref{cor:upper_bound} is generally very weak, but its strength lies in its universality. In particular, it implies that  $\tau^{(r)}(H)=o(n^r)$ for \emph{all} $s$-graphs $H$. In some cases, however, it is not so bad.

\begin{example}\label{Ex4}\rmfamily
Consider the clique $H:=K_6^{(3)}$ and $r=2$. We have $e_H=20$ and $\Delta_1(H)=10$. So, we set $\lambda:=20-10=10$ and, remembering that the  right-hand-side of \eqref{eqn:ell} in this case is just $k-1$, apply Corollary~\ref{cor:upper_bound} with $k=11$. As a result, we obtain the bound $\tau^{(2)}(H)\le n^{2-\tfrac2{11}}$, not so far from the lower bound $n^{ 2-\tfrac14}$ established in Theorem~\ref{thm:lower_bound_general} and even closer to the correct bound $n^{ 2-\tfrac29}$ from~\cite{ShD}.
\end{example}

The ideas used in the proof of Theorem~\ref{thm:upper_bound_balanced} can be extended to cover families of $s$-graphs violating assumption \eqref{eqn:ell-gen}, but the obtained upper bounds do not match the lower bounds in Theorem~\ref{thm:lower_bound_general}. They are, however, better than those established in Corollary~\ref{cor:upper_bound}. In Section~\ref{sec:K_6} we prove such bounds for general cliques. (Note that this theorem is also true in the case $r=1$ yielding, however, a worse bound than the optimal Theorem~\ref{thm:r=1}.)

\begin{restatable}{theorem}{cliques}
\label{thm:general_cliques}
Given $2\le r<s\le k$, let $\ell:=\ell_k(r,s)$ be the smallest integer  such that
\begin{equation}\label{constraint}
k-\ell-r-\frac{k-\ell}{\binom ks-\binom{\ell}{s}}\sum_{j=1}^r\binom\ell{s-j}\left[\binom{k-\ell}j-\binom rj\right]\le0.
\end{equation}
Then there exists a strategy $\mathcal S$ such that for
$$t\gg n^{r-\frac{k-\ell}{\binom ks-\binom{\ell}s}}$$ a.a.s.\ $K_k^{(s)}\subset G_t^{(r,s)}$. Thus,
$$\tau^{(r)}(K_k^{(s)})\le n^{r-\frac{k-\ell}{\binom ks-\binom{\ell}s}}.$$
\end{restatable}

In fact, the conclusion of Theorem~\ref{thm:general_cliques} remains true for \emph{any} $\ell$ satisfying~\eqref{constraint}. However, as shown in the Appendix (see~\eqref{mono}), the exponent $r-\tfrac{k-\ell}{\binom ks-\binom{\ell}s}$ is an increasing function of $\ell$, so the best upper bound on $\tau^{(r)}(K_k^{(s)})$ is, indeed,  obtained for $\ell=\ell_k(r,s)$. Moreover, $\ell=k-r$ satisfies~\eqref{constraint}, so $\ell_k(r,s)$ is well-defined and $\ell_k(r,s)\le k-r$.

Observe also that $\ell_k(r,s)\ge s-r$, since otherwise the left-hand-side of~\eqref{constraint} would be equal to  $k-\ell-r > k-(s-r)-r=k-s>0$.
Moreover, $\ell_k(r,s)=s-r$  if and only if
$$k-s-\frac{k-s+r}{\binom ks}\left[\binom{k-s+r}r-1\right]\le0$$
which is equivalent to  inequality~\eqref{clique_starplus}, stated after Corollary~\ref{thm:upper_bound_complete}, characterizing those cliques for which the upper bound of Theorem~\ref{thm:upper_bound_balanced} matches the lower bound of Theorem~\ref{thm:lower_bound_general}. Indeed, we see that in that case the  conditions on $t$  in Theorems~\ref{thm:general_cliques} and~\ref{thm:upper_bound_balanced} are the same.

It is not easy, in general, to compute $\ell_k(r,s)$. We only managed to show (see Appendix) that
\begin{equation}\label{elka}
\ell_k:=\ell_k(2,3)=\left\lceil k+\frac32-\sqrt{6k+1/4}\right\rceil.
\end{equation}
In the next smallest case we were only able to get the asymptotic lower bound $\ell_k(2,4)=k-\Omega(\sqrt k)$.

\begin{example}\label{Ex5}\rmfamily
Set $s=3$ and $r=2$ and note that $\ell_6=\ell_7=2$. Thus, we get upper bounds $n^{9/5}$ and $n^{13/7}$ for $\tau^{(2)}(H)$ where $H$ is, respectively, $K_6^{(3)}$ and $K_7^{(3)}$, which are not far from the lower bounds $n^{7/4}$ and $n^{64/35}$ given by Theorem~\ref{thm:lower_bound_general}.
We also have $\ell_8=3$, so the threshold for $K_8^{(3)}$ is squeezed between $n^{105/56}$ and $n^{107/56}$.
The values of $\ell_k$ grow rapidly with $k$, getting closer and  closer to $k$ in ratio. Already $\ell_{20}=11$.
\end{example}

\subsection{Better lower bounds}\label{blb}

Finally, let us identify examples of hypergraphs $H$ for which $\tau^{(r)}(H)$ is of a strictly greater order of magnitude than the lower bound given by Theorem~\ref{thm:lower_bound_general} or its corollary. In fact, we have an infinite family of them. We will find them within a class of hypercycles which we define now, along with the corresponding paths.

 An \textbf{$\ell$-tight $s$-uniform cycle} $C_m^{(s,\ell)}$ is an $s$-graph with $k=(s-\ell)m$ vertices and $m$ edges which are formed by segments of consecutive vertices  evenly spread along a cyclic ordering of the vertices in such a way that consecutive edges overlap in exactly $\ell$ vertices.
(E.g., the set of triples $123,345,567,781$ forms a copy of $C_4^{(3,1)}$ and $12345,34567,56781,78123$ forms a copy of $C_4^{(5,3)}$.)
Note that $m\ge \lfloor (s+1)/(s-\ell)\rfloor$ and that non-consecutive edges may also overlap (if $\ell>s/2$).  An \textbf{$\ell$-tight $s$-path} $P_m^{(s,\ell)}$ with $m\ge1$ edges is defined similarly. It has exactly $(s-\ell)m+\ell$ vertices.  In particular,  $C_m^{(s,s-1)}=C_m^{(s)}$ and $P_m^{(s,s-1)}=P_m^{(s)}$ are the tight cycle and tight path defined earlier. (Often, 1-tight cycles and path are called \emph{loose}.) See Figure~\ref{fig:ell-tight} for an example of $C^{(5,3)}_6$ and $P^{(5,3)}_4$.

\begin{figure}[h]
    \centering
    \includegraphics[width=0.8\linewidth]{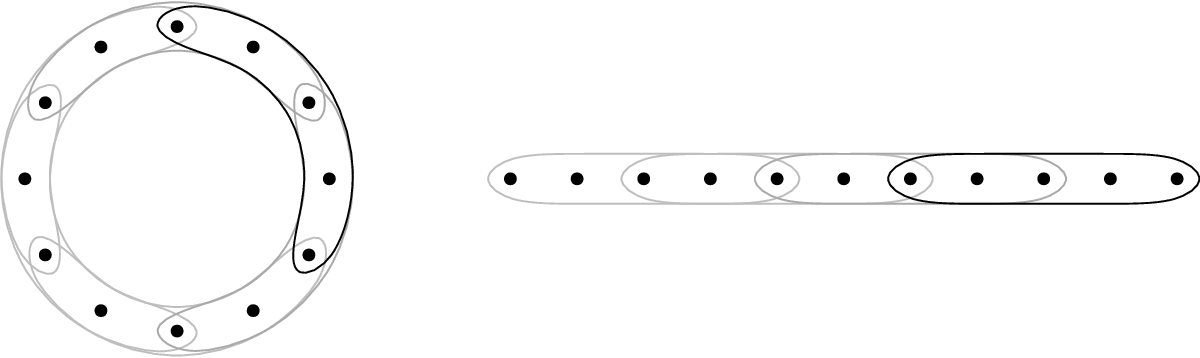}
    \caption{A $3$-tight $5$-uniform cycle $C^{(5,3)}_6$ and a $3$-tight $5$-uniform path  $P^{(5,3)}_4$.}
    \label{fig:ell-tight}
\end{figure}



Let us first summarize what we already know about the threshold function  $\tau^{(r)}$ for $\ell$-tight paths and cycles based on Theorem~\ref{thm:r=1}, Theorem~\ref{thm:lower_bound_general}, and Corollary~\ref{cor:lower_bound}. Since the degeneracies are $d(P_m^{(s,\ell)})=1$ and $d(C_m^{(s,\ell)})=\lfloor\frac s{s-\ell}\rfloor$ (see Appendix, Claim~\ref{degen}), Theorem~\ref{thm:r=1} yields $\tau^{(1)}(P_m^{(s,\ell)})=1$ and, in particular,
$$\tau^{(1)}(C_m^{(s,\ell)})=\begin{cases} 1\quad\mbox{for}\quad \ell<s/2,\\
\sqrt n\quad\mbox{for}\quad \ell=s/2,\\
n^{1-1/s}\quad\mbox{for}\quad \ell=s-1.
\end{cases}
$$
By monotonicity, the above quantities set also lower bounds for $\tau^{(r)}(C_m^{(s,\ell)})$, $r\ge2$.

Next, let us have a closer look at the lower bounds in Theorem~\ref{thm:lower_bound_general} and Corollary~\ref{cor:lower_bound} in the context of $\ell$-tight paths and cycles (all calculations are deferred to Appendix, Claim~\ref{miu}). We have
$$\mu^{(r)}(P_m^{(s,\ell)})=\begin{cases}\frac1r \quad\mbox{for}\quad s-r\ge\ell\\\frac{m}{(s-\ell)m+\ell-s+r} \quad\mbox{otherwise}.
\end{cases}
$$
Thus, for $s-r\le\ell-1$,
$$\tau^{(r)}(P_m^{(s,\ell)})\ge  n^{r+\ell-s+\frac{s-r-\ell}m},
$$
while for the remaining values of $r$, we get the trivial bound of 1.

For $\ell$-tight cycles one can show that
\begin{equation}\label{miucycle}
\mu^{(r)}(C_m^{(s,\ell)})=\max\left\{\frac m{(s-\ell)m-s+r},\frac 1r \right\}=\begin{cases}
\frac1r \quad\mbox{for}\quad s-r\ge2\ell
\\\frac m{(s-\ell)m-s+r}\quad\mbox{for}\quad s-r\le\ell.
\end{cases}
\end{equation}
Thus, for $ s-r\le\ell$,
$$\tau^{(r)}(C_m^{(s,\ell)})\ge  n^{r+\ell-s+\frac{s-r}m}.
$$

For $\ell +1\le s-r\le2\ell-1$, however, the formula for $\mu^{(r)}(C_m^{(s,\ell)})$ depends on how large $m$ is: more precisely, it is $\frac 1r$ for $m\ge\tfrac{s-r}{s-\ell-r}$ and $\frac m{(s-\ell)m-s+r}$ otherwise. These are the cases within which the lower bound on $\tau^{(r)}(C_m^{(s,\ell)})$  can be improved. Indeed, with some extra assumptions on $\ell$ and $r$, in Section~\ref{s6} we are able to prove the following.


\begin{restatable}{proposition}{rsl}\label{prop:rsl}
Let $m\ge3$, $s\ge3$, $1\le\ell\le s/2$, and $s-r\ge\ell$.
Then, $\tau^{(r)}(P_m^{(s,\ell)})=1$ while
$$
\tau^{(r)}(C_m^{(s,\ell)})
\begin{cases}
=1\qquad\qquad&\mbox{if}\quad s-r\ge2\ell
\\=n^{1/2}\quad\quad&\mbox{if}\quad s-r=2\ell-1
 \\= n^{\frac{r-s+2\ell}3}\quad\quad&\mbox{if}\quad s-r\le2\ell-2.
 \end{cases}
$$
   \end{restatable}
\noindent Notice that  all thresholds stated in Proposition~\ref{prop:rsl} are independent of $m$, the number of edges. Moreover, the thresholds $\tau^{(r)}(C_m^{(s,\ell)})$ for $\ell\le s-r\le2\ell-1$ are higher than the lowers bound in Corollary~\ref{cor:lower_bound}, except when $m=3$ \emph{and} $s-r=\ell$.
To see this, note that $1/2>1/3=(r-s+2\ell)/3$ for $ s-r=2\ell-1$ and, in general, the inequality
$$\frac{r-s+2\ell}3\ge r+\ell-s+\frac{s-r}m$$
is equivalent to $(s-r)(2m-3)\ge\ell m$, which holds for $m\ge3$, since $s-r\ge\ell$, and is strict for $m\ge4$ or $s-r\ge\ell+1$.
Hence, for $s\ge3$, $1\le\ell\le s/2$, $m\ge3$, and $s-2\ell+1\le r\le s-\ell$, with the above mentioned exception,
$$\tau^{(r)}(C_m^{(s,\ell)})>n^{r-1/ \mu^{(r)}(C_m^{(s,\ell)})},$$
improving the lower bound from Corollary~\ref{cor:lower_bound}.
The smallest instances in this class, with $r=2$, are $C_4^{(4,2)}$ --- for which the two lower bounds on $\tau^{(2)}$ are, respectively, $n^{1/2}$ and $n^{2/3}$, and $C_3^{(5,2)}$ --- with the two lower bounds, $1$ and $n^{1/2}$. The first one can be generalized: by Theorem~\ref{prop:rsl}, 
for all $m\ge4$, we have $\tau^{(2)}(C_m^{(4,2)})=n^{2/3}$, while the lower bound in Corollary~\ref{cor:lower_bound} is $n^{2/m}$.

\subsection{Probabilistic tools}\label{prelim} Here we gather some elementary probabilistic facts and estimates to be used later throughout the proofs. We begin with a version of the second moment method, useful for so called \emph{counting} random variables, where the variance is being expressed in terms of the second factorial moment. Let $Y$ be a nonnegative, integer-valued random variable. Then, for every $\epsilon>0$, Chebyshev's inequality gives
\begin{equation}\label{2ndMM}
\Prob\left(|Y-\E Y|\ge\epsilon\E Y\right)\le\frac{\mathrm{Var}(Y)}{\epsilon^2(\E Y)^2}=\frac1{\epsilon^2}\left(\frac{\E(Y(Y-1))}{(\E Y)^2}+\frac1{\E Y}-1\right).
\end{equation}
Assuming that $\E Y\to\infty$ as $n \to \infty$, in order to show that the above probability tends to 0, it suffices to show that $\E(Y(Y-1))\sim(\E Y)^2$.

Next, we give an estimate of the probability that the random multi-$r$-graph $R_{t}^{(r)}:=\{U_1,\dots,U_t\}$ contains a fixed sub-multi-$r$-graph. More precisely, let $F$ be a multi-$r$-graph with the vertex set $V(F)\subset [n]$, $h$ vertices, $m$ edges, and  multiplicities $m_1,\dots,m_{\binom  hr}$ (some of which may be equal to 0). We want to estimate $\Prob(F\subset R_t^{(r)})$.

Consider
 first a small example. Let $r=2$ and $T$ be the triangle on a fixed vertex set $\{1,2,3\}$  with the edge $\{1,2\}$ doubled, that is, the multiplicities are $2,1,1$. If one insists that the  times of hitting particular edges  are fixed, say, at $1\le t_{13}<t_{12}< t_{23}<t_{12}'\le t$, then the probability of actually creating $T$  at these designated times is precisely
\begin{align*}&\left(1-\frac3{\binom n2}\right)^{t_{13}-1}\times\frac1{\binom n2}\times\left(1-\frac2{\binom n2}\right)^{t_{12}-t_{13}-1}\times\frac1{\binom n2}\times
\\&\left(1-\frac2{\binom n2}\right)^{t_{23}-t_{12}-1}\times\frac1{\binom n2}\times\left(1-\frac1{\binom n2}\right)^{t_{12}'-t_{23}-1}\times\frac1{\binom n2}\sim \left(\frac1{\binom n2}\right)^4,
\end{align*}
 as long as $t = o(n^2)$. The number of ways to select the four hitting times and assign them to the four edges, due to the exchangeability of $t_{12}$ and $t_{12}'$, is $\binom{t}4\times 4!/2$. Thus,
$$\Prob(T\subset R_t^{(2)})\sim \binom{t}4\frac{4!}2 \left(\frac1{\binom n2}\right)^4\sim\frac{1}2 \left(\frac{t}{\binom n2}\right)^4.$$
Similarly, in the general case, setting $p=t/\binom nr$,
\begin{equation}\label{multi}
\Prob(F\subset R_t^{(r)})\sim\frac{p^{m}}{m_1!\cdots m_{\binom hr}!}.
\end{equation}

The proof of Proposition~\ref{prop:rsl}  uses the following simple lemma.

\begin{lemma}\label{trice}
For all $\lfloor x/2\rfloor\le q\le r$, a.a.s.\ every $q$-element subset of $[n]$ is contained in at most three sets $U_i$, $i=1,\dots,t$.
\end{lemma}
\proof Let $X$ be the number of $q$-element sets contained in at least four sets  $U_i$, $i=1,\dots,t$.
For a fixed $q$-element set $Q$, the probability that it is contained in at least four sets  $U_i$ is $O(t^4/n^{4q})$. Thus, as there are $\binom nq<n^q$ such sets and $4x/3<3\lfloor x/2\rfloor\le3q$ for $x\ge2$,
$\E X=O(t^4/n^{3q})=\omega^4n^{4x/3-3q}=o(1)$.
\qed

For the proof of Theorem~\ref{thm:upper_bound_balanced}   we will need the following lemma which is a straightforward generalization of Theorem 3.29 from~\cite{JLR} (the edge-disjoint case) to hypergraphs. Given integers $2\le r<n$, a real $p:=p(n)\in(0,1)$, and an $r$-graph $F$, let
$$\Phi_F:=\Phi_F(n,p)=\min\{n^{v_{F'}}p^{e_{F'}}:\; F'\subseteq F,\; e_{F'}>0\}$$ 
be the order of magnitude of the expectation of the ``least expected'' sub-hypergraph of $F$. This quantity, in turn, determines the order of magnitude
of the largest number $D_F:=D_F(n,p)$ of edge-disjoint copies of $F$ one can find in $G^{(r)}(n,p)$, a random $n$-vertex $r$-graph obtained by turning each $r$-element subset of vertices into an edge independently with probability $p$.
\begin{lemma}[Janson, {\L}uczak, Ruci\'{n}ski \cite{JLR}]\label{disj} For all integers $r\ge2$ and every $r$-graph $F$, there exist constants $0<a<b$ such that
if $\Phi_F\to\infty$, then  a.a.s. $a\Phi_F\le D_F\le b\Phi_F$. 
\end{lemma}

Lemma~\ref{disj} will be used to facilitate the desired outcome of Phase 1 of a strategy supporting the proof of Theorem~\ref{thm:upper_bound_balanced}.
However, in order to apply this lemma in our context, we need first to address two issues: (i) the appearance of repeated edges and (ii) the uniformity of the model -- as opposed to the binomial model  $G^{(r)}(n,p)$. For the latter
we will use a consequence of an asymptotic model  equivalence result   from \cite[Corollary 1.16(i)]{JLR}. Let $G^{(r)}(n,t)$ be an $r$-graph chosen uniformly at random from all $r$-graphs on vertex set $[n]$ which have $t$ edges.

\begin{lemma}[Janson, {\L}uczak, Ruci\'{n}ski \cite{JLR}]\label{equi} For all integers $r\ge2$ and every increasing property $Q$ of $r$-graphs, if  $G^{(r)}(n,p)$ has $Q$ a.a.s., then $G^{(r)}(n,t)$ also has $Q$ a.a.s., provided $p=t/\binom nr$.
\end{lemma}

The issue of repeated edges can be resolved by taking an appropriate random subsequence of the process $(R_t^{(r)})_t$.

\begin{lemma}\label{no-rep} For all integers $r\ge2$ and every sequence $t:=t(n)=o(n^r)$, there is a joint distribution of the  random multi-$r$-graph $R_t^{(r)}$ and the random equiprobable  $r$-graph $G^{(r)}(n,t')$, where $t' = t-t^{3/2}/n^{r/2}=t-o(t)$ such that a.a.s.\ $G^{(r)}(n,t')\subset R_t^{(r)}$.
\end{lemma}

\proof The expected number of times a repetition occurs in $(R_t^{(r)})_t=(U_1,\dots,U_t)$ (that is, an edge is selected again) is at most $t\times t/\binom nr=o(t)$. Thus, by Markov's inequality, a.a.s.\ there are no more than $t^{3/2}/n^{r/2}$  such times. This means that  along with $(R_t^{(r)})_t$ one can a.a.s.\ generate its sub-process $(U_{j_1},\dots,U_{j_{t'}})$ with $t'$ ``unparalleled''  edges. Indeed, just ignore a chosen edge whenever it had been chosen before. Then the next edge, provided it is not ignored, is selected uniformly at random from those $r$-tuples of vertices that are not present already. We may identify the sub-process $(U_{j_1},\dots,U_{j_{t'}})$ with the ``static'' equiprobable $r$-graph $G^{(r)}(n,t')$. Finally, note that when $t=o(n^r)$, we have $t^{3/2}/n^{r/2}=o(t)$. \qed

\medskip

Lemmas~\ref{equi} and~\ref{no-rep} together imply a swift transition between our model and the standard binomial model.

\begin{corollary}\label{L+L}
 Let $r\ge2$ and $Q$ be an increasing property of $r$-graphs. Further, let $t:=t(n)=o(n^r)$, $t'=t-t^{3/2}/n^{r/2}$, and $p:=p(n)=t'/\binom nr$. If $G^{(r)}(n,p)$ has $Q$ a.a.s., then  $R_t^{(r)}=\{U_1,\dots,U_t\}$ also has $Q$ a.a.s.
\end{corollary}

\proof If $G^{(r)}(n,p)$ has $Q$ a.a.s., then, by Lemma~\ref{equi}  $G^{(r)}(n,t')$ also has $Q$ a.a.s. By Lemma~\ref{no-rep}, a.a.s.\ $R_t^{(r)}$ contains a copy of $G^{(r)}(n,t')$ and thus, by the monotonicity of $Q$, it too possesses $Q$ a.a.s.
\qed

\medskip

Before turning to the proof of Theorem~\ref{thm:upper_bound_balanced}, we need to show a simple fact.

\begin{claim}\label{simple} If an $r$-graph $F$ is edge-balanced and $p=o(n^{-1/g(F)})$, then $\Phi_F=n^{v_F}p^{e_F}$.
\end{claim}
\proof
First observe that for all $F'\subseteq F$ with $v_F>v_{F'}>r$ the inequality $g(F')\le g(F)$ implies that $g(F)\le\frac{e_F-e_{F'}}{v_F-v_{F'}}$. Thus,
$$n^{v_F-v_{F'}}p^{e_F-e_{F'}}=\left(np^{\frac{v_F-v_{F'}}{e_F-e_{F'}}}\right)^{v_F-v_{F'}}\le\left(np^{g(F)}\right)^{v_F-v_{F'}}=o(1),$$
which yields that, indeed, $\Phi_F=n^{v_F}p^{e_F}$.
\qed

\section{Proof of Theorem~\ref{thm:lower_bound_general}}\label{s3}

Here we prove Theorem~\ref{thm:lower_bound_general}, restated below for convenience.

\LBgeneral*

\begin{proof} Set $G_t:=G_t^{(r,s)}$ and let $H$ be an $s$-graph with $k$ vertices and $m$-edges. This generic proof relies on an obvious observation that for a copy of $H$ to exist in $G_t$, there must be, in the first place, a set of $k$ vertices spanning at least $m$ edges of $G_t$. Formally, for any $j$, $1 \le j \le m$, any time $t$, and any strategy $\mathcal{S}$, let $X_{j}^{\mathcal{S}}(t)$ be a random variable counting the number of $k$-element sets of vertices that induce in $G_t$ at least $j$ edges at the end of round $t$.
We will assume that the player plays according to a strategy $\mathcal{S}$. However, since we only provide a universal upper bound for the expected value of $X_{j}^{\mathcal{S}}(t)$, it will actually not depend on $\mathcal{S}$. Therefore, to unload the notation a little bit, let us suppress the dependence on the strategy $\mathcal{S}$.

We will show by induction on $j$ that for any $1 \le j \le m$ we have
\begin{equation}\label{eq:lower_bound_induction}
\E X_{j}(t) \le t^{j} k^{r(j-1)} n^{k-s+r - rj}\qquad\qquad\mbox{for all }\;t\ge1.
\end{equation}
The base case $j=1$, holds trivially and deterministically, as we  have
$$
X_{1}(t) \le t \binom{ n-s}{k-s } \le t n^{k-s}.
$$
Indeed, there are precisely $t$ edges at the end of round $t$, and each of them is contained in  $\binom{n-s}{ k-s }$ sets of size $k$ (as we are after an upper bound, we ignore the possible repetitions of the $k$-sets here).

For the inductive step, suppose that~(\ref{eq:lower_bound_induction}) holds for some value of $j-1$, $1 \le j-1 < m$ (and all $t$) and our goal is to show that it holds for $j$ too (again, for all $t$).
 We say that a set $W \subseteq [n]$, $|W|=k$, is of \emph{type $j$ at time $t$} if it spans in $G_t$ at least $j$ edges and we define an indicator random variable $I_{j}^{W}(t)$  equal to 1 if $W$ is of type $j$ at time $t$, and 0 otherwise. Thus (as a sanity check),
 \begin{equation}\label{XI}
 X_{j}(t)=\sum_{W\in\binom{[n]}k}I_{j}^{W}(t).
 \end{equation}

Note that in order to \emph{create} a set $W$ of type $j$ at time $i$, it is necessary that $W$ was of type $j-1$ at time $i-1$ (in fact, having exactly $j-1$ edges), as well as,  the $r$-vertex set selected by the semi-random hypergraph process at time $i$ is contained in $W$, that is, $U_i \subseteq W$.
Thus, setting also $J^{W}(t)=1$ if $U_t\subseteq W$ and 0 otherwise, we have
\begin{equation}\label{XIJ}
X_{j}(t)\le \sum_{i=1}^{t}\sum_{W\in \binom{[n]}k}I_{j-1}^{W}(i-1)J^{W}(i).
\end{equation}
Since $U_i $ is selected uniformly at random from $\binom{[n]}r$,
$$\E(J^{W}(i))=\Prob(J^{W}(i)=1)=\frac{\binom kr}{\binom nr}\le\frac{k^r}{n^r},$$
as $k\le n$.

We now take the expectation on both sides of \eqref{XIJ}.   Using the linearity of expectation, and the independence of $I_{j-1}^{W}(i)$ and $J^{W}(i)$,  \eqref{XI}, 
we get that
\begin{eqnarray*}
\E X_{j}(t) &\le& \sum_{i=1}^{t} \ \E\left(\sum_{W\in\binom{[n]}k}I_{j-1}^{W}(i-1)\right)\frac{k^r}{n^r} ~~=~~ \sum_{i=1}^{t} \ \E X_{j-1}(i-1) \frac{k^r}{n^r}\\
&\le& \sum_{i=1}^t \left(i^{j-1} k^{r(j-1-1)} n^{k-s+r - r(j-1)}\right) \cdot \frac{k^r}{n^r} ~~\le~~ t^{j} k^{r(j-1)} n^{k-s+r - rj},
\end{eqnarray*}
and so (\ref{eq:lower_bound_induction}) holds for $j$ too. This finishes the  inductive proof of \eqref{eq:lower_bound_induction}.

The desired conclusion is now easy to get. Note that, by~\eqref{eq:lower_bound_induction} with $j=m$,
$$
\E X_{ m}(t) \le t^{m} k^{r({m}-1)} n^{k-s+r - r{m}} = O \left(  n^{k-s+r} \left( \frac {t}{n^{r}} \right)^{m} \right).
$$
Hence, if $t = o( n^{r-(k-s+r)/{m}} )$, then $\E X_{ m }(t) = o(1)$ and so, by Markov's inequality, $X_{ m }(t) = 0$ a.a.s. Since the presence of a copy of $H$ in $G_t$ implies that $X_{ m }(t) \ge 1$, we conclude that a.a.s.\ $G^{(r,s)}_t \not\in \mathcal{P}_H$ which was to be proved.
\end{proof}

\bigskip

\section{Proof of Theorem \ref{thm:upper_bound_balanced}}\label{s4}
In this section we prove Theorem~\ref{thm:upper_bound_balanced}, restated here for convenience.

\UBbalanced*

\medskip

\begin{proof}[Proof of Theorem~\ref{thm:upper_bound_balanced}] For  integers $s>r\ge2$, let $H$ be an $(s,s-r)$-starplus on $k \ge s$ vertices with $\lambda_1$ rays and excess
$\lambda_2$, satisfying the assumptions of Theorem \ref{thm:upper_bound_balanced}.
Set
$$m=|E(H)|=\lambda_1+\lambda_2\;,\quad\kappa=r-\frac{k-s+r}{m}\;,\quad\mbox{and}\quad t=\omega n^{\kappa},$$
 where $\omega:=\omega(n)\to\infty$ as $n \to \infty$ but, say, $\omega=o(\log n)$.

Let us again abbreviate $G_t:=G_t^{(r,s)}$. To play the game ${\mathcal P}_H$, we equip the player with the following strategy. The vertex set $C=\{1,\dots,s-r\}$ is put aside. From the player's point of view there will be two phases of the game (but just one for $\lambda_2=0$), lasting, respectively,
$t_1$ and $t_2 := t-t_1$ steps, where
$$t_1=\begin{cases}
t & \mbox{ when }\lambda_2=0 \\
t/2 & \mbox{ when \eqref{eqn:ell-gen} is strict}\\
t/\omega_1 & \mbox{ when there is equality in \eqref{eqn:ell-gen}},
\end{cases}$$ where
$$\omega_1=\omega^{\frac{\lambda_1+\lambda_2/2}{\lambda_1}}.$$
For convenience, we also set
$$p=\frac{t}{\binom nr},\quad p_1=\frac{t_1}{\binom nr},\quad \text{ and } \quad p_2=\frac{t_2}{\binom nr}.$$


During Phase 1, whenever a random $r$-set $U_i$ lands within $[n]\setminus C$, the player draws the edge $U_i\cup C$, that is, they choose $V_i=C$.
 The goal of this phase is to collect  sufficiently many
 \emph{edge-disjoint} copies of $H_1$ on $[n]\setminus C$ created purely by the random $r$-sets $U_i$ of $R_t^{(r)}$.
 According to the player's strategy, this will yield in $G_{t_1}$ plenty of copies of the $s$-uniform $c$-star on $k$ vertices with the same center $C$ whose flowers are isomorphic to $H_1$. This will end the proof when $\lambda_2=0$. In fact, in this special case all we need is just one copy of $H_1$.

So, let us start with the special case $\lambda_2=0$. Then,
$$p_1\sim r!\omega n^{-\frac{k-s+r}m}=r!\omega n^{-\frac{v_{H_1}}{e_{H_1}}}.$$ Moreover, since $H_1$ is edge-balanced, it is also balanced (see Appendix, Claim \ref{ebal->bal}) in the usual sense, that is, $e_{H'_1}/v_{H'_1}\le e_{H_1}/v_{H_1}$ for all sub-$r$-graphs $H_1'$ of $H_1$. Thus, it can be routinely shown by the second moment method (cf. the proof of the 1-statement of \cite[Theorem 3.4]{JLR}) that a.a.s. $G^{(r)}(n,p_1')$, where $p_1'=t'/\binom nr$ and $t'$ are given in Corollary \ref{L+L}, contains a copy of $H_1$ vertex-disjoint from $C$. (The expected number of copies of $H_1$ containing at least one vertex of $C$ is  $O(\omega^{\lambda_1}/n)=o(1)$.) By Corollary \ref{L+L}, the same property is a.a.s. satisfied by $R_t^{(r)}$, which completes the proof in this case.

\medskip

From now on, assume that $\lambda_2\ge 1$.
Recall that $H_1$ is the flower of $H$ and note that $g(H_1)=\frac{\lambda_1-1}{k-s}$. So,  if \eqref{eqn:ell-gen} is strict, then
$$p_1=\frac p2\sim\frac{r!}2\omega n^{-\frac{k-s+r}{\lambda_1+\lambda_2}}=o(n^{-1/g(H_1)}).$$
On the other hand, if there is equality in \eqref{eqn:ell-gen}, then
$$p_1=\frac{\omega n^{-\frac{k-s+r}{\lambda_1+\lambda_2}}}{\omega_1}\overset{\eqref{eqn:ell-gen}}{=}\frac{n^{-1/g(H_1)}}{\omega^{\lambda_2/2\lambda_1}}=o(n^{-1/g(H_1)}),$$
again. Thus, in either case we have
\begin{equation}\label{p1}
p_1=o(n^{-1/g(H_1)}).
\end{equation}

By Lemma \ref{disj} (noting that $\Phi_{H_1} = n^{k-s+r}p_1^{\lambda_1} = \Theta\left(n^{(k-s+r)(1-\lambda_1/m)}\right) \rightarrow \infty$), Claim \ref{simple}, Equation \eqref{p1}, and  Corollary \ref{L+L}, for some $a>0$, there is a.a.s.\ a family $\mathcal H''$ of  edge-disjoint copies of $H_1$ in $R_t^{(r)}$ of size $|\mathcal H''|\sim an^{k-s+r}p_1^{\lambda_1}$.  The expected number of copies of $H_1$ intersecting $C$ is $O(|\mathcal H''|/n)$, so, after deleting them, we obtain a family $\mathcal H'$ of  edge-disjoint copies of $H_1$, all of which are vertex-disjoint from $C$, of asymptotically the same size as $\mathcal H''$. Also, crucially, the expected number of pairs of edge-disjoint copies of $H_1$ which share at least $r$ vertices is, by \eqref{eqn:ell-gen} and the definition of $t_1$, $O(|\mathcal H'|^2/n^r)=o(|\mathcal H'|)$. Thus, by further deleting from $\mathcal H'$ one copy of each such pair, we obtain the ultimate family $\mathcal H$ of  edge-disjoint copies of $H_1$ which avoid $C$, and  pairwise share fewer than $r$ vertices, whose size is
$$J:=|\mathcal H|\sim an^{k-s+r}p_1^{\lambda_1}.$$

Let $\mathcal H=\{H_1^{(1)},\dots,H_1^{(J)}\}$.
 By the player's strategy,
each $H_1^{(i)}$ forms in $G^{(r,s)}_{t_1}$ the flower of an  $(s-r)$-star $S^{(i)}$ with center $C$.
 In order to turn one of them into a copy of the starplus $H$, during Phase 2, it has to be hit  $\lambda_2$ times by the random $r$-sets which, collectively, should be extendable (by the player) to a copy of~$H_2$, the cap of $H$, and thus create a copy of $H$. For simplicity, we assume that the $\lambda_2$ $r$-sets are to be contained in the $H_1^{(i)}$, that is, disjoint from $C$.

To this end, as a preparation, for each  $i=1,\dots, J$, we designate a multi-$r$-graph $M_i$ of $\lambda_2$  $r$-element subsets of $V(S_1^{(i)})$ (with possible, and sometimes necessary, repetitions) such that their suitable extensions to $s$-sets lead to a copy $H_2^{(i)}$ of $H_2$. This can be easily done by selecting (in a template copy of $H_2$) one $r$-element subset of each edge of $H_2$, disjoint from $C$. See Figure~\ref{fig:starplus} for one example; as another, more abstract example, consider an instance where the flower $H_2\supseteq  K_6^{(3)}$ and $r=2$ -- clearly, some of the 15 pairs of the vertices of the clique must appear in $M$ more than once (as there are 20 edges to be covered).

Since the $H^{(i)}_1$'s share pairwise fewer than $r$ vertices, the families $M_i$ are pairwise disjoint, so there is no ambiguity for the player. During Phase 2, whenever a random $r$-set $U_i$ lands on one of the $r$-sets in $M_i$ for some $i$, the player draws the corresponding $s$-edge (within $V(S_1^{(i)})$) and gradually builds a copy of $H$.

 \begin{figure}[ht]
    \centering
    \includegraphics[scale=.8]{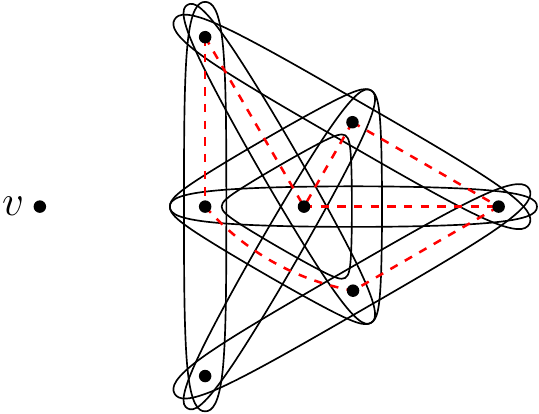}
    \caption{A 3-uniform starplus on 8 vertices with surplus edges forming the Fano plane (the 21 edges containing $v$ are not shown). The dashed-line red pairs indicate a possible choice of the graph $M$ (one of $3^7$). In this particular case, obviously,  $M$ cannot be a multigraph.}
    \label{fig:starplus}
\end{figure}

Next, we move to a detailed description of Phase 2 of the process which lasts $t_2:=t-t_1$ steps.
 Set $R^{(r)}_{t_2}$ to be the random $r$-graph  consisting of the random $r$-sets $U_i$, $i=t_1+1,\dots,t$. Thus, $R^{(r)}_{t_2}$  adds random $r$-edges to a fixed, typical instance of $G^{(r,s)}_{t_1}$. Further, let $I_i=1$ if $M_i\subseteq R_{t_2}$ and $I_i=0$ otherwise. Then, our goal is to prove that a.a.s.\  $Y:=\sum_{i=1}^{J}I_i>0$. Unlike in phase one, we cannot rely on Corollary \ref{L+L}, as $M$ may be a multigraph. Instead, we apply the second moment method, as described in Subsection \ref{prelim},  to $Y$ along with the estimate \eqref{multi}.

 By symmetry, the expectation $\E I_i=\Prob(M_i\subseteq R_{t_2})$ is the same for all $i$.
Denoting by $m_1,\dots,m_{q}$ the  multiplicities of the $r$-sets of vertices in $M$, where $m_1+\cdots+m_{q}=\lambda_2$ and $q=\binom{k-s+r}r$, we obtain, using \eqref{multi} and the definitions of $J$ and $t_2$,
\begin{equation}\label{multiEY}\E Y=J\E I_1\sim J\frac{p_2^{\lambda_2}}{m_1!\cdots m_{q}!}=\begin{cases}
       \Theta\left( \omega^{m}\right) \quad \mbox{ when \eqref{eqn:ell-gen} is strict}    \\
   \Theta\left( \omega^{\lambda_2/2}\right) \mbox{ when there is equality in \eqref{eqn:ell-gen},}                        .                                   \end{cases}
   \end{equation}
    where $m=\lambda_1+\lambda_2$.
In fact, the choice of $\omega_1$ has been driven by this very calculation.

 Again by symmetry,
$$\E(Y(Y-1))=J(J-1)\Prob(I_1=I_2=1).$$
 Since  the families $M_1$ and $M_2$  are edge-disjoint (and the number of common vertices does not matter), similarly as above, applying \eqref{multiEY} to $M_1\cup M_2$, we get the estimate
$$\E(Y(Y-1))\sim J^2\frac{p_2^{2\lambda_2}}{m_1!^2\cdots m_{q}!^2}\sim (\E Y)^2.$$
So, by \eqref{2ndMM} with $\epsilon=1/2$, a.a.s.\ $Y\ge \tfrac12\E Y>0$, which completes the proof. \end{proof}

 \section{Proof of Theorem \ref{thm:general_cliques}} \label{sec:K_6}

We restate Theorem~\ref{thm:general_cliques} below for convenience.

\cliques*

The proof of Theorem \ref{thm:general_cliques} relies on a bold extension of the strategy used in the proof of Theorem~\ref{thm:upper_bound_balanced}.
Since the details are quite technical, we decided to present the argument gently, beginning with the smallest open case, $K_6^{(3)}$, then outline the proof for all cliques $K_k^{(3)}$, $k\ge 6$ and $r=2$, before finally moving to the general case.  For an $r$-graph $H$ and a natural number $m$, we denote by $mH$ the multi-$r$-graph obtained by
replacing every edge of $H$ by $m$ parallel edges.

\subsection{The clique $K_6^{(3)}$}\label{clic}
Although $\tau^{(2)}(K_6^{(3)})$ has been already determined in~\cite{ShD}, we use this special case as a gentle introduction to the general proof of Theorem \ref{thm:general_cliques}. To this end, we prove the following weaker bound.

 \begin{proposition}\label{K6}
$\tau^{(2)}(K_6^{(3)})\le n^{9/5}$
 \end{proposition}
 \begin{proof}
 Consider the following version of the strategy used in the proof of Theorem~\ref{thm:upper_bound_balanced}. In essence, we alter the way the edge set of the target hypergraph $K_6^{(3)}$ is split between Phase~1 and Phase~2. Although $s-r=3-2=1$, we put aside not one but two vertices, say $n-1$ and $n$.
  Set $t=\omega n^{9/5}$ where $\omega =\omega(n) \rightarrow \infty$ with $n$ and  $\omega \le \log{n}$ say, and set $t_1=t_2=t/2$,  $p=t/\binom n2$ and $p_i=t_i/\binom n2$, $i=1,2$.

  In Phase 1 the \emph{first} time a $2$-element subset $U$ of $[n-2]$ is randomly selected,  it is extended by the player to the $3$-edge $U\cup\{n-1\}$, while if $U$ is selected for the second time, it is extended to
$U\cup\{n\}$. In addition, whenever a random pair $U$ contains $n-1$ but not $n$, it is extended to the triple $U\cup\{n\}$.
 So, in $R_{t_1}=G^{(2,2)}_{t_1}$, we are after  double cliques $2K_4$ with vertex sets in $[n-2]$, rooted at $n-1$ with the root connected by a single edge to all four vertices of the double clique (see Figure~\ref{fig:clique}). Let us denote such a graph by $2K_4^{+1}$.  By player's strategy, each copy of $2K_4^{+1}$ in $R^{(2)}_{t_1}$ yields a copy of $K_6^{(3)}-K_4^{(3)}$ in $G^{(2,3)}_{t_1}$.

 \begin{figure}
     \centering
     \includegraphics[scale=.9]{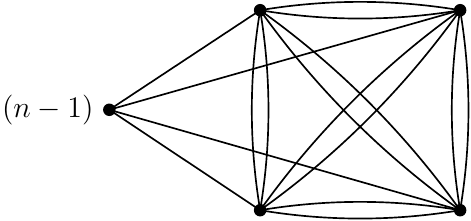}
     \caption{The graph $2K_4^{+1}$}
     \label{fig:clique}
 \end{figure}

 As in the previous proof, we would like to show that a.a.s.\ there are many copies of $2K_4^{+1}$ in $R^{(2)}_{t_1}$ which pairwise share at most one vertex from $[n-2]$. We cannot, however, apply the approach presented before and based, in particular, on Corollary~\ref{L+L}, because now we are counting copies of \emph{rooted multi-graphs}  $2K_4^{+1}$. Let $X$ be the number of copies of $2K_4^{+1}$ in $R_{t_1}$.
By the second moment method, we are going to show that a.a.s.\  $X=\Theta(\E X)=\Theta(n^4p_1^{16})$.  In doing so, we follow the technique described in Section~\ref{prelim}.

By~\eqref{multi},
$$\E X\sim \binom{n-2}4\frac{p_1^{16}}{2^6}=\Theta(n^4p_1^{16})=\Theta(\omega^{16}n^{4/5}).$$ To estimate $\E X(X-1)$, we split all pairs of distinct copies of $2K_4^{+1}$ in $K_n$ according to the size $g$ of their non-rooted vertex-intersection (disregarding the root $n-1$).  Then, by~\eqref{multi},
$$\E X(X-1)\sim\sum_{g=0}^3\binom{n-2}{8-g}\binom{8-g}{(4-g)!^2g!}\frac{p_1^{32-g^2}}{2^{12-\binom g2}}\sim\sum_{g=0}^3 \frac{(2\omega)^{32-g^2}}{(4-g)!^2g!2^{12-\binom g2}} n^{8-g-(32-g^2)/5}.$$
Denoting the four summands above by $S_g$, $g=0,1,2,3$, we see that $S_0\sim(\E X)^2=\Theta(n^{8/5})$, while $S_1=\Theta(n^{4/5})$ and $S_2,S_3=\Theta(n^{2/5})$. In conclusion, by~\eqref{2ndMM}, a.a.s.\ $X=\Theta(\E X)$ as claimed.

The above estimates, in addition, imply that the expected number of pairs of copies of $2K_4^{+1}$ in $R^{(2)}_{t_1}$, which share at least one double edge is $n^{2/5}=o(n^{4/5})$. Hence, removing one copy from each such pair, we obtain a.a.s a family $\mathcal G$ of $\Theta(n^{4/5})$ copies of $2K_4^{+1}$ in $R^{(2)}_{t_1}$, which pairwise share at most one vertex other than the root $n-1$.

To see what happens in Phase 2, consider the double clique contained in one of the copies of $2K_4^{+1}$ belonging to $\mathcal G$, say, on vertices $1,2,3,4$. In order to turn the corresponding copy of $K_6^{(3)}-K_4^{(3)}$ into a copy of $K_6^{(3)}$, one needs to add to it four edges -- the four 3-element subsets of $[4]$.
 This can be facilitated by the following strategy: when during the process $R^{(2)}_{t_2}$ a pair $\{j,j+1\}$ is hit, $j=1,\dots,4$, the player extends it to $\{j,j+1,j+2\}$ (here $5:=1$ and $6:=2$), that is, by adding the next vertex along the cycle $12341$. In the notation of the proof of Theorem~\ref{thm:upper_bound_balanced}, we thus have $M=\{12,23,34,14\}$ and $|\mathcal G|$ designated copies of the 4-cycle $M$ at the end of Phase 1.
By~\eqref{multi}, the expected number of those  of them which will be hit in Phase 2 is $\Theta(n^4p_1^{16}p_2^{4})=\Theta(\omega^2)\to\infty$. Again, by the second moment method (details, similar to those at the end of the proof of Theorem~\ref{thm:upper_bound_balanced}, are omitted) a.a.s.\ at least one of them will be present in $G_t$, completing, per player's strategy, a copy of $K_6^{(3)}$.
\end{proof}

\subsection{Larger 3-uniform cliques ($r=2$)}

Here we prove Theorem~\ref{thm:general_cliques} still in the case  $r=2,s=3$, but for \emph{all} $k$. We singled out this special case, because it is the only one in which we may express the result explicitly. Indeed, as proved in the Appendix (around inequality~\eqref{equiv}), for $s=3$, $r=2$, and every $k$, the smallest integer  which satisfies~\eqref{constraint} is given by~\eqref{elka}, that is,
$$\ell_k:=\ell_k(2,3)=\left\lceil k+\frac32-\sqrt{6k+1/4}\right\rceil.$$

\begin{proposition}\label{Kk} For every $k\ge4$,
$$\tau^{(2)}(K^{(3)}_k)\le n^{r-\frac{k-\ell_k}{\binom k3-\binom{\ell_k}3}},$$
where $\ell_k$ is as above.

 \end{proposition}
 \begin{proof}[Proof (outline)]
  The proof is by induction on $k$. The need for induction comes from a new phase of the player's strategy, Phase 0, when a copy of the clique $K^{(3)}_{\ell_k}$ is built.
  It follows from Theorem~\ref{thm:upper_bound_balanced} and Proposition~\ref{K6} that the statement is true for $k\le6$ (with $\ell_4=\ell_5=1$ and $\ell_6=2$), so let $k\ge7$. For ease of notation we put $\ell:=\ell_k$.
  To facilitate induction, set
  $$\bar\ell=\left\lceil \ell+\frac32-\sqrt{6\ell+1/4}\right\rceil,$$
   and observe that by the monotonicity of function $f_s(k,\ell)$ (see Appendix, the comment after the proof of~\eqref{mono})
  \begin{equation}\label{less}
  \frac{\ell-\bar\ell}{\binom{\ell}3-\binom{\bar\ell}3}>\frac{k-\ell}{\binom k3-\binom{\ell}3}.
  \end{equation}

   Set
  $$t=\omega n^{2-\frac{k-\ell}{\binom k3-\binom{\ell}3}}\quad\mbox{and}\quad t_0=\omega n^{2-\frac{\ell-\bar\ell}{\binom{\ell}3-\binom{\bar\ell}3}},$$
  for $\omega =\omega(n) \rightarrow \infty$ with $n$ and $\omega \le \log{n}$, say,
  and note that, by~\eqref{less}, $t_0=o(t)$.
  Further, set $t_1=t/\omega_1$, where
  $$\omega_1=\omega^{\frac{\binom k3-\binom{\ell}3-\binom{k-\ell}3/2}{\binom k3-\binom{\ell}3-\binom{k-\ell}3}}.$$
  Finally, set $t_2=t-t_0-t_1$, and  $p_i=t_i/\binom n2$, $i=0,1,2$.

 We split the game into three phases. In the preliminary Phase 0 which lasts $t_0$ steps, we produce a.a.s.\ a copy of $K^{(3)}_\ell$ by the induction's hypothesis applied to $\ell$. Fix one such copy with vertex set $L$.  Without loss of generality, we may assume that $L=[\ell]$.

In Phase 1, which lasts $t_1$ steps,  whenever a 2-element subset $U$ of $[n]\setminus L$ is randomly selected for the $i$-th time, $i=1,\dots,\ell$, we extend it to the triple $U\cup\{i\}$.
Moreover, whenever a 2-element subset $U$ of the form $U=\{u,j\}$, where $u\in[n]\setminus L$ and $j\in L$, is randomly selected for the $i$-th time, $i=1,\dots,\ell-j$, we extend it to the triple $U\cup\{j+i\}$ (see Figure~\ref{fig:clique_phase1}).

\begin{figure}[h]
    \centering
    \includegraphics[scale=.8]{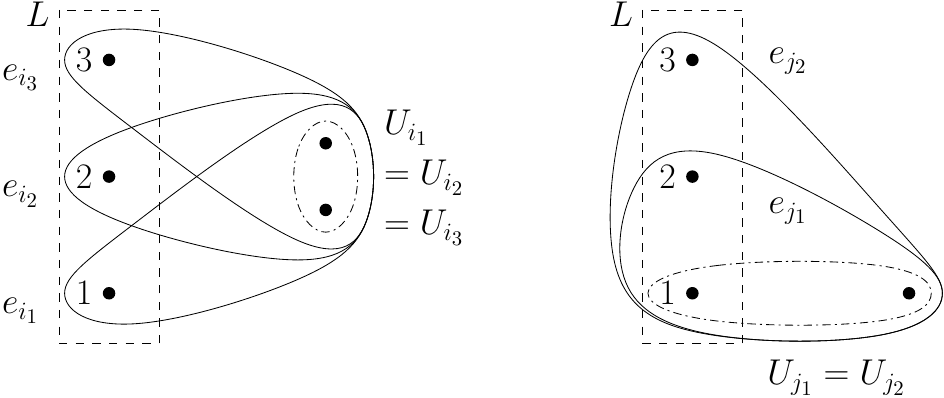}
    \caption{How two different sets $U$ are extended to edges each time they are randomly selected in Phase 1, where $\ell = 3$, $i_1 < i_2 < i_3$ and $j_1 < j_2$.}
    \label{fig:clique_phase1}
\end{figure}

Thus, our goal is to produce many copies of the rooted multi-graph $F$ consisting of a multi-clique  $\ell K^{(2)}_{k-\ell}$ with vertices in $[n]\setminus L$ and, for each vertex $u\in[n]\setminus L$, of $(\ell-1)+\cdots+1=\binom\ell2$ extra edges connecting $u$ with $L$ in such a way that the multiplicity of the edge $uj$ is $\ell-j$, $j=1,\dots,\ell$ (see Figure~\ref{fig:bigger_clique} for an example in the case $k=9$ and $\ell=3$). By the player's strategy, a copy of $F$ in $R^{(2)}_{t_1}$ corresponds to a copy of $K_k^{(3)}-K_{k-\ell}^{(3)}$ in $G_{t_0+t_1}$.

\begin{figure}[h]
    \centering
    \includegraphics[scale=.8]{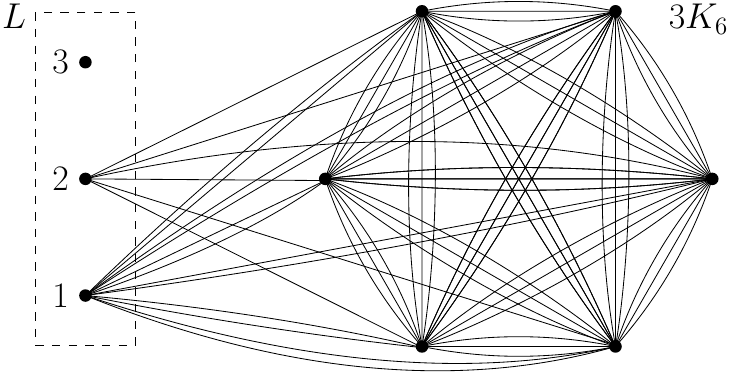}
    \caption{The multigraph $F$ when $k=9$ and $\ell=3$}
    \label{fig:bigger_clique}
\end{figure}

Let $X$ be the number of copies of $F$ in $R^{(2)}_{t_1}$.  Setting
$$h=k-\ell,$$
and noting that
$$\ell\binom{h}2+\binom{\ell}2h=\binom k3-\binom\ell3 -\binom{h}3,$$
we have, by~\eqref{multi},
$$\E X\sim\binom{n-\ell}{h}p_1^{\ell\binom{h}2+\binom{\ell}2h}=
\Theta\left(\left(\frac{\omega}{\omega_1}\right)^{-\frac12\binom h3}n^{\frac{h\binom h3}{\binom k3-\binom{\ell}3}}\right)\to\infty.$$

 By the second moment method, we will soon show that a.a.s.\ there are $\Theta(\E X)$ copies of $F$ at the end of Phase~1. But crucially, we need that, as before, most of them are edge-disjoint (within $[n]\setminus L$), to avoid ambiguity in Phase~2. ``To kill two birds with one stone'', we will estimate quantities $S_g$, $g=0,\dots,h-1$, defined as expected numbers of ordered pairs of copies of $F$ in $R^{(2)}_{t_1}$ which share $g$ vertices outside of $L$. Note that
$\E X(X-1)=\sum_{g=0}^{h-1}S_g$, while $\frac12\sum_{g=2}^{h-1}S_g$ is the expected number of pairs of copies of $F$ which share at least one pair of vertices outside $L$. We aim at showing that
\begin{equation}\label{EE2}
\E X(X-1)\sim S_0\sim(\E X)^2
\end{equation}
and
\begin{equation}\label{S2}
\sum_{g=2}^{h-1}S_g=o(\E X).
\end{equation}
 We have
$$S_g\sim\binom{n-\ell}{2h-g}\binom{2h-g}{h-g,h-g,g}p_1^{\ell\left(2\binom{h}2-\binom g2\right)+\binom\ell2(2h-g)}.$$
Thus,
$$S_0\sim\frac{n^{2h}}{h!^2}p_1^{2\left(\ell\binom{h}2+\binom\ell2h\right)}\sim(\E X)^2.$$
Next,
$$S_1=O\left(n^{2h-1}p_1^{2\left(\ell\binom{h}2+\binom\ell2h\right)-\binom\ell2}\right)=O\left(\frac{(\E X)^2}{np_1^{\binom\ell2}}\right)=o\left((\E X)^2\right),$$
as $np_1^{\binom\ell2}\to\infty$ (since $\ell\le k-2$). Now comes the critical $S_2$. We claim that, by the definition of $\ell$ and~\eqref{constraint},
$$S_2=O\left(n^{2h-2}p_1^{\ell\left(2\binom{h}2-1\right)+\binom\ell2(2h-2)}\right)=o(\E X),$$
equivalently,
$$n^{h-2}p_1^{\ell\left(\binom{h}2-1\right)+\binom\ell2(h-2)}=o(1).$$
Indeed, if there is a strict inequality in~\eqref{constraint}, then the left-hand-side above is of the order $\Theta(n^{-\epsilon})$ for some $\epsilon>0$.
Otherwise the polynomial term disappears and we are looking at
$$\left(\frac{\omega}{\omega_1}\right)^{\binom k3-\binom{\ell}3-\binom{h}3-\ell^2}.$$
However, 
$$\binom k3-\binom{\ell}3-\binom{h}3-\ell^2>0,$$
is equivalent to $\ell\le k-3$ which is true for $k\ge7$ (see Appendix,~\eqref{k-2}). Thus, in this case $S_2=o(\E X)$ as well.

 The same is true for $g=3,\dots,h-1$, which can be demonstrated by induction on~$g$. Assume that for some $2\le g\le h-2$,
$$S_g=\Theta\left(n^{2h-g}p_1^{\ell\left(2\binom{h}2-\binom g2\right)+\binom\ell2(2h-g)}\right)=o(\E X),$$
equivalently,
$$ n^{h-g}p_1^{\ell\left(\binom{h}2-\binom g2\right)+\binom\ell2(h-g)}=o(1).$$
However,  the  equation above can be rewritten as
$$\left(np_1^{\ell(k+g-2)/2}  \right)^{h-g} =o(1),$$
which implies that
$$np_1^{\ell(k+g-2)/2}=o(1).$$  This, in turn, implies that
$$np_1^{\ell(k+(g+1)-2)/2}=o(1)$$
(as, trivially, $p_1=O(1)$), and, consequently,
$$\left(np_1^{\ell(k+(g+1)-2)/2}  \right)^{h-g-1}=o(1),$$
which, by the same token as above, is equivalent to $S_{g+1}=o(\E X)$. Thus, we have proved~\eqref{EE2} and~\eqref{S2}. Consequently, by~\eqref{2ndMM} with, say $\epsilon=1/2$, a.a.s.\ $X=\Theta(\E X)$, and, more importantly, by standard removal, we obtain a.a.s.\ a family $\mathcal F$ of $\Theta(\E X)$ copies of $F$ in $R_{t_1}$ which pairwise share at most one vertex outside $L$. As mentioned earlier, each copy of $F$ yields a copy of $K_k^{(3)}-K_{k-\ell}^{(3)}$ in $G_{t_0+t_1}$.

In Phase 2, a.a.s.\ the player's goal is to extend at least one of them to a copy of $K_k^{(3)}$. This will be possible if a copy of $F$ is hit by the random pairs of $R_{t_2}$ at least  $\binom{h}3$ times and onto appropriate spots. To this end, let $M$ be a multi-graph obtained by selecting one pair of vertices from each triple of $K_{k-\ell}^{(3)}$. For each $F'\in\mathcal F$, let $M'$ be a copy of $M$ on $V(M')\setminus L$. Let $\mathcal M$ be the family of such copies of $M$ and let $Y$ be the number of them present in $R_{t_2}$. Then, by~\eqref{multi},
$$\E(Y)=\Theta\left(\E X\times p_2^{\binom{h}3}\right)=\Theta\left( n^hp_1^{\ell\binom{h}2+\binom{\ell}2h}p_2^{\binom{h}3} \right)=\Theta\left( \frac{\omega^{\binom{k}3-\binom{\ell}3}}{\omega_1^{\binom{k}3-\binom{\ell}3-\binom h3}}  \right)=\Theta\left(\omega^{\tfrac12\binom h3}\right),$$
which goes to $\infty$ as $n$ goes to $\infty$.
 Finally, one can easily show, again by the second moment method, that a.a.s.\ $Y>0$, which means that the player can indeed create a copy of $K_k^{3}$ in $G_t^{(2,3)}$.
 \end{proof}

\subsection{General cliques}\label{genclic}
In this subsection we prove Theorem~\ref{thm:general_cliques} in its full generality. Experienced with the proofs presented in the two previous subsections, we just outline here how to extend them to arbitrary $2\le r<s$. Given $r,s$, $k$, and $\ell:=\ell_k(r,s)$, so far the general scheme for the player has been to build a desired clique in  three big chunks: $K_\ell^{(s)}$ (Phase 0, vacuous when $\ell<s$), $K_{\ell,k-\ell}^{(s)}$ (Phase 1), and $K_{k-\ell}^{(s)}$ (Phase 2).  We basically follow that suit in the general case, with the border between Phases 1 and 2 refined.

\begin{proof}[Proof of Theorem \ref{thm:general_cliques} (outline)] We  proceed by induction on $k\ge s$, with $r$ and $s$ fixed, $2\le r<s$. The base of induction, the case $k=s$ is trivial (then $\ell_s(r,s)=s-r$). Fix $k>s$ and assume the statement is true for all $s\le k'<k$.
Let $\ell=\ell_k(r,s)\ge s-r$ be the smallest integer satisfying~\eqref{constraint} and
$$t=\omega n^{r-\frac{k-\ell}{\binom ks-\binom{\ell}s}}.$$
If $\ell<s$, we set $L=[\ell]$ and skip Phase 0.
Otherwise,  let $\bar\ell$ stand for the smallest integer $\bar\ell$ satisfying~\eqref{constraint} with $k$ and $\ell$ replaced, respectively, by $\ell$ and $\bar\ell$.
Phase 0 will last
$$t_0=\omega n^{r-\frac{\ell-\bar\ell}{\binom{\ell}s-\binom{\bar\ell}s}}$$
steps. Again, $t_0=o(t)$, by the monotonicity of $\frac{k-\ell}{\binom ks-\binom{\ell}s}$.
Since $s\le\ell\le k-r<k$,
by the induction assumption we a.a.s. get and fix a copy  of $K_\ell^{(s)}$ whose vertex set we denote by $L$.
Without loss of generality, set $L=[\ell]$.

Let $H_1$ be the sub-$s$-graph of $K_k^{(s)}$ consisting of all edges with at least $s-r$ but fewer than $s$ vertices in a fixed $\ell$-element vertex subset $L_0$. Further, let $H_2$  be the sub-$s$-graph of $K_k^{(s)}$ consisting of all edges with fewer than $s-r$ vertices in $L_0$ (see Figure~\ref{fig:L0}). Observe that $K_k^{(s)}=K_\ell^{(s)}\cup H_1\cup H_2$. Moreover,
$$|H_1|=\sum_{j=1}^{r}\binom{k-\ell}j\binom{\ell}{s-j}\quad\mbox{and}\quad|H_2|=\sum_{j=r+1}^{s}\binom{k-\ell}j\binom{\ell}{s-j},$$
so $|H_1|+|H_2|=\binom ks-\binom{\ell}s$, as it should. Set $\eta_i=|H_i|$, $i=1,2$, for convenience. If $\ell=k-r$, then $\eta_2=0$ and no second phase is needed. Indeed, we then take $t_1=t-t_0$ and  a.a.s.\ find a copy of $H_1$ in $G^{(r,s)}_{t_1}$ by the second moment method and the player's strategy described below.

\begin{figure}[h]
    \centering
    \includegraphics[scale=0.8]{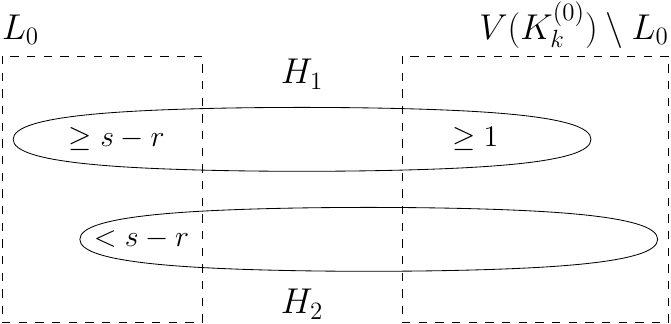}
    \caption{The edges of $K_k^{(s)}$ lying in $H_1$ and $H_2$.}
    \label{fig:L0}
\end{figure}

Otherwise, that is, when $\ell<k-r$,
we take $t_1=t/\omega_1$, where
$$\omega_1= \omega^{\frac{\eta_1+\eta_2/2}{\eta_1}}.$$
In Phase 1, we are going to build many copies of $H_1$ in $R_{t_1}^{(r)}$  with the set $L_0$ mapped onto $L$ (order preserving). In Phase 2 at least one of them will be extended  by a copy of $H_2$ to form an ultimate copy of $K_k^{(s)}$.

The player's strategy in Phase 1 is, thus, as follows. Set $j_0:=\max\{1,s-\ell\}$.
For every $j_0\le j\le r$, assign to each $(s-j)$-element subset $S$ of $L$ one of its $(r-j)$-element subsets and denote it by  $T_S$. Note that for $j=r$ all $(s-r)$-element sets $S$ are assigned the empty set, that is, $T_S=\emptyset$. Given an $(r-j)$-element subset $T$ of $L$, let  $S_T^{(1)},\dots,S_T^{(m_T)}$ be all sets $S\in\binom L{s-j}$ for which $T=T_S$. Observe that for some $T$ we may have $m_T=0$,  and that $\sum_{T\in \binom{L}{r-j}}m_T=\binom{\ell}{s-j}$. Whenever a random $r$-set $U$ is hit for the $i$-th time, $i=1,\dots, m_{U\cap L}$, the player extends it to the $s$-set $U\cup S_{U\cap L}^{(i)}$.

Thus, in order for the player to generate  a copy of $H_1$ rooted at $L$, the random $r$-sets in $R_{t_1}^{(r)}$ must form a copy of the $r$-graph $F$, rooted at $L$, which consists of $k-\ell$ vertices in addition to $L$, and such that every $r$-element subset $e$ of vertices in $F$ has multiplicity $m_{e\cap L}$.
 Note that $F$ has the same number of edges as $H_1$, that is, $|F|=\eta_1$. Let $X$ be the number of copies of $F$ in $R_{t_1}^{(r)}$. Then, by~\eqref{multi}, letting $h = k-\ell$,
 $$\E X=\Theta\left(n^{k-\ell}p_1^{\eta_1}\right)=\Theta\left(\left(\frac{\omega}{\omega_1}\right)^{\eta_1}n^{h-\frac{h\eta_1}{\eta_1+\eta_2}} \right)\to\infty,$$
 because the exponent of $n$ is positive, while $\omega$ and $\omega_1$ are at most logarithmic in $n$.
 Now, by a standard second method one can show that a.a.s.\  $X=\Theta(n^{k-\ell}p_1^{\eta_1})$. Moreover, similarly to the proof of Proposition~\ref{Kk}, one can  show that most of the copies of $F$ share pairwise fewer than $r$ vertices outside $L$. Indeed, setting as before $S_g$, $g=0,\dots,h-1$, for the expected numbers of ordered pairs of copies of $F$ in $R^{(r)}_{t_1}$ which share $g$ vertices outside $L$, we have
 $$S_g\sim\binom{n-\ell}{2h-g}\binom{2h-g}{h-g,h-g,g}p_1^{2\eta_1-\sum_{j=1}^r\binom gj\binom{\ell}{s-j}}.$$
 Hence, $S_0\sim(\E X)^2$ and, for $g\le r$,
 $$S_g\sim\frac{(\E X)^2}{n^gp_1^{\sum_{j=1}^g\binom gj\binom{\ell}{s-j}}}=o((\E X)^2),$$
 since $(h/g)\sum_{j=1}^g\binom gj\binom{\ell}{s-j}<\eta_1$ and, recall, $\binom ks-\binom{\ell}s=\eta_1+\eta_2$. For $g\ge r$, however, we need a stronger bound on $S_g$, namely, that $S_g=o(\E X)$, or equivalently,
 $$T_g:=n^{h-g}p_1^{\sum_{j=1}^r\binom{\ell}{s-j}\left\{\binom hj -\binom gj \right\}}=o(1).$$
 This can be shown by induction on $g$, $g=r,\dots,h-1$. Let $\alpha$ denote the  left hand side of~\eqref{constraint}, that is,
 $$\alpha=k-\ell-r-\frac{k-\ell}{\binom ks-\binom{\ell}{s}}\sum_{j=1}^r\binom\ell{s-j}\left[\binom{k-\ell}j-\binom rj\right].$$
 For $g=r$,
 $$T_r=n^\alpha(\omega/\omega_1)^{\eta_1-\sum_{j=1}^r\binom{\ell}{s-j}\binom rj}.$$ If $\alpha<0$, we are done. Otherwise, observe that $\omega=o(\omega_1)$ while $\eta_1>\sum_{j=1}^r\binom{\ell}{s-j}\binom rj$ as $r<h$, and we are done again. Now assume that for some $r\le g\le h-2$ we have $T_g=o(1)$. To proceed, we rewrite $T_g$ as
 $$T_g=n^{h-g}p_1^{\sum_{j=1}^r\binom{\ell}{s-j}\frac{h-g}{j!}f_j(h,g)},$$
 where
 $$f_j(h,g)=\sum_{i=1}^j\frac{(h)_j-(g)_j}{h-g}$$ is an increasing function of $g$ (see Appendix,~\eqref{mono}).
 As $T_g=o(1)$ implies that
 $$np_1^{\sum_{j=1}^r\binom{\ell}{s-j}\frac{1}{j!}f_j(h,g)}=o(1),$$
 we  also have
 $$\left(np_1^{\sum_{j=1}^r\binom{\ell}{s-j}\frac{1}{j!}f_j(h,g+1)}\right)^{h-g-1}=o(1),$$
 which is equivalent to $T_{g+1}=o(1)$, or $S_{g+1}=o(\E X)$.
 We conclude that at the end of Phase 1, there is a.a.s.\  a family $\mathcal F$ of $\Theta(n^hp_1^{\eta_1})$ copies of $F$ every two of which share fewer than $r$ vertices outside $L$. Every copy $F'$ of $F$, via the player's strategy, corresponds to a copy $H_1'$ of $H_1$ in $G_{t_1}^{(r,s)}$.

 In Phase 2, to turn a copy $H_1'$  into a copy of $K_k^{(s)}$, we still need to place onto it
$$\eta_2=\sum_{j=r+1}^s\binom{k-\ell}j\binom{\ell}{s-j}$$
extra edges forming  a copy $H_2'$ of $H_2$.  These are the edges with at least $r+1$ vertices outside $L$, so the player can create them all from random $r$-edges $U$ of $R_{t_2}^{(r)}$ falling onto the $L$-free part of a copy of $F$. Similarly as in Phase 1, we assign to each $s$-edge $S$ of $H_2$ an $r$-element subset $T_S$ disjoint from $L_0$, and the sets $S_T^{(i)}$, $i=1,\dots, m_T$, are defined as before. This way we obtain a multi-$r$-graph $M$ whose edge multiplicities sum up to $\eta_2$. For each $F'\in\mathcal F$, let $M'$ be a copy of $M$ on $V(M')\setminus L$, and let $\mathcal M$ be the family of all those copies of $M$. Further,  let $Y$ be the number of the copies of $M$ in $\mathcal M$ present in~$R^{(r)}_{t_2}$. Then, by~\eqref{multi},
$$\E Y=\Theta\left(n^hp_1^{\eta_1}p_2^{\eta_2}\right)=\Theta\left(\frac{\omega^{\binom ks-\binom{\ell}s}}{\omega_1^{\eta_1}} \right)\to\infty.$$
Finally, by the second moment method, one can routinely show that a.a.s.\ $Y>0$ and, consequently,  the player will create a copy of $K_k^{(s)}$ by the end of Phase 2. The player's strategy is straightforward again: whenever a random $r$-edge $U\subset V(M')$ is drawn for the $i$-th time, $i=1,\dots, m_U$, the player extends it to the $s$-edge $S_U^{(i)}\subseteq V(F')$  which has been assigned to $U$.
\end{proof}

\section{Proof of Proposition~\ref{prop:rsl}}\label{s6}

In this section we prove Proposition~\ref{prop:rsl}, repeated below for convenience.

\rsl*

Due to the assumption $\ell\le s/2$, non-consecutive edges of an $\ell$-tight $s$-cycle and $s$-path are disjoint. Also, for cycles, every edge has exactly $s-2\ell$ vertices of degree 1 and $2\ell$ vertices of degree 2, while in paths of length at least~2, one can distinguish two edges, each with $s-\ell$ vertices of degree 1 and $\ell$ vertices of degree 2. We refer to them as the \textbf{end-edges} of the path.

Recall  that in the $i$-th step of the semi-random process,  $U_i$ is a random $r$-element subset selected uniformly from all $r$-element subsets of $[n]$. Thus, for any fixed subset $T\subseteq[n]$, by Bernoulli's inequality,
\begin{equation}\label{Tbound}\Prob(U_i\cap T\neq\emptyset)=1-\frac{\binom{n-|T|}r}{\binom nr}\le1-\left(\frac{n-|T|}n\right)^r\le\frac{r|T|}n=O(|T|/n).
\end{equation}
Set $e_i=U_i\cup V_i$ for convenience, and  assume throughout that $s-r\ge\ell$.

\begin{proof}[Proof of  Proposition~\ref{prop:rsl}]

\subsubsection*{Case $H=P_m^{(s,\ell)}$.}

We equip the player with the following strategy. The player will grow just one copy of $P_m^{(s,\ell)}$ beginning with $e_1$ and extending it whenever the next random edge $U_i$ is disjoint from the so far built path. If this happens, then one constructs the set $V_i$, and consequently the whole edge $e_i$, by including in it $\ell$ vertices of degree one belonging to an end-edge of the current path, and any $s-\ell-r$ ``fresh" vertices, that is, not belonging to the current path (see Figure~\ref{fig:path}). Otherwise, the player ``wastes'' the move by doing whatever.
The probability of failure in at least one of the first $m$ steps is, by~\eqref{Tbound},  $O(1/n)=o(1)$ and so the player can complete a copy of $H$ a.a.s.\ in just $m$ steps. Thus, we have $\tau^{(r)}(H)=1$.

\begin{figure}
    \centering
    \includegraphics[scale=.7]{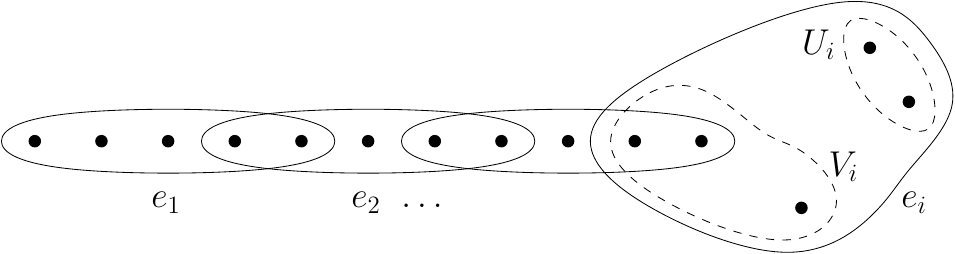}
    \caption{An example of how to build the path $P_m^{(5,2)}$ when $r=2$.}
    \label{fig:path}
\end{figure}

\subsubsection*{Case $H=C_m^{(s,\ell)}$, $s-r\ge2\ell$.}  A similar strategy also works in this case. The player first constructs a path $P:=P_{m-1}^{(s,\ell)}$, as described above. Then, in the $m$-th and final step, provided $U_m\cap V(P)=\emptyset$, the player composes $V_m$ of  $\ell$ vertices of degree one from each end-edge of $P$ and any $s-2\ell$ fresh vertices. The probability of failure is, again, $o(1)$. Thus, $\tau^{(r)}(H)=1$.

\subsubsection*{Case $H=C_m^{(s,\ell)}$, $s-r\le2\ell-1$, lower bound} To establish the desired lower bound on $\tau^{(r)}(H)$, notice that no matter how the game progresses, in order to achieve a copy of $H$, the final edge $e_i$, $i\ge m$, has to connect the two end-edges, say $e'$ and $e''$, of a copy of $P_{m-1}^{(s,\ell)}$ built so far. As the player can only contribute $s-r<2\ell$ vertices to $e_i$, the random set $U_i$ must draw at least $2\ell-(s-r)$ vertices from the one-degree vertices of $e'$ and $e''$.

We consider separately the case $s-r=2\ell-1$, since then $U_i$ needs just one vertex from $e'\cup e''$. Since at time $i$ there are $i-1$ edges, the probability of $U_i$ intersecting at least one of them is, by~\eqref{Tbound}, $O(i/n)$. Summing over all times $i\le t$, the probability that this will happen by time $t$ is $O(t^2/n)=o(1)$ whenever $t=o(n^{1/2})$. This proves that $\tau^{(r)}(H)\ge n^{1/2}$ in this special case.

When $s-r\le2\ell-2$, $U_i$ may hit both, $e'$ and $e''$, so we need to consider all \emph{pairs} of edges.
Since, at any time $i$, there are $\binom{i-1}2<i^2$ pairs of edges, the probability of $U_i$ hitting at least $r-s+2\ell$ vertices from the union of one pair is
$$O\left(i^2\times \frac{n^{r-(2\ell-s+r)}}{n^r}\right)=O\left(\frac{i^2}{n^{r-s+2\ell}}\right).$$ Summing over all times  $i\le t$, the probability that this will happen by time $t$ is $O\left(t^3/n^{r-s+2\ell}\right)$, which is $o(1)$ for any $t=o\left(n^{\frac{r-s+2\ell}3}\right)$. This proves that $\tau^{(r)}(H)\ge n^{\frac{r-s+2\ell}3}$.

\subsubsection*{Case $H=C_m^{(s,\ell)}$, $s-r=2\ell-1$.}  We have $\tau^{(r)}_{H}\ge n^{1/2}$ and want to prove a matching upper bound.
Let $t=\omega n^{1/2}$, where $\omega:=\omega(n)\to\infty$ arbitrarily slowly.
This time we propose a more sophisticated player's strategy which consists of three phases. In Phase 0 we build, a.a.s.\ in just $m-3$ steps  the path $P=P_{m-3}^{(s,\ell)}$ of length $m-3$, as described earlier in this proof. At this point we see already that the smallest cases $m=3$ and $m=4$ are somewhat special. Indeed, for $m=3$ Phase 0 is vacuous, while for $m=4$ it consists of just one step as we take $P=\{e_1\}$.

For $m\ge5$, in each end-edge of $P$ we fix a set of $\ell$ vertices of degree one and call these sets  $L'$ and $L''$. For $m=4$, we take $L',L''\subset e_1$, $L'\cap L''=\emptyset$. For $m=3$, $L'=L''=:L$ is an arbitrary fixed subset of $[n]$ of size $\ell$.

In Phase 1, which lasts $t_1:=\lfloor(t-m+3)/2\rfloor$ steps, the player in alternating time steps creates a set $E'$ of $t':=\lfloor t_1/3\rfloor$   edges containing $L'$, and a set $E''$ of $t'$  edges containing $L''$, whose sets of new vertices are disjoint from $V(P)\setminus(L'\cup L'')$ as well as from each other. (For $m=3$, the player creates $2t'$ edges containing $L$ but otherwise mutually disjoint.)

This is feasible, because the probability that a random $r$-set $U_i$ is \emph{not} disjoint from all previously built edges is, by~\eqref{Tbound}, $O(t_1/n)$. Thus, the expected number of such ``failed'' steps is $O(t_1^2/n) = O(\omega^2)$, and so, a.a.s.\  at least $t_1-\omega^3\ge t'$ sets $U_i$ drawn in Phase 1 are disjoint from all previously built edges. Each time such a $U_i$ arrives, the player extends it to an $s$-edge $e_i$ by including in $V_i$ the set $L'$ for $i$ odd and to $L''$ for $i$ even, while the remaining $\ell-1$ vertices of $V_i$ are to be ``fresh'', that is, not belonging to any previous edge (see Figure~\ref{fig:cycle_phase1}).

\begin{figure}[h]
    \centering
    \includegraphics[scale=0.7]{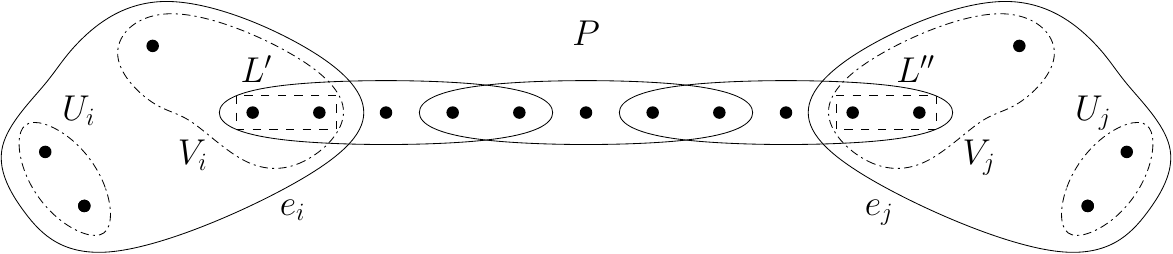}
    \caption{Phase 1 of building $C_6^{5,2}$ for $r=2$, where $i$ is odd and $j$ is even.}
    \label{fig:cycle_phase1}
\end{figure}

In Phase 2, lasting  $t_2=t_1$ steps, the player waits until a random set $U_i$ satisfies $U_i\cap V(P)=\emptyset$ (for $m=3$, $U_i\cap L=\emptyset$) and, for some edges $e'\in E'$ and $e''\in E''$, we have $|U_i\cap e'|=1$ and $|U_i\cap e''|=0$ (note that the existence of $e''$ satisfying the second condition is trivially guaranteed by the disjointness of edges in $E''$, since $t'>r$). Once this happens, a copy of $H$ can be created by including in $V_i$  $\ell-1$ vertices from $e'\setminus L'$ and $\ell$ vertices from $e''\setminus L''$ (see Figure~\ref{fig:cycle_phase2}). Then, $P$ together with edges $e',e_i,e''$ form a copy of $H$.

\begin{figure}[h]
    \centering
    \includegraphics[scale=0.7]{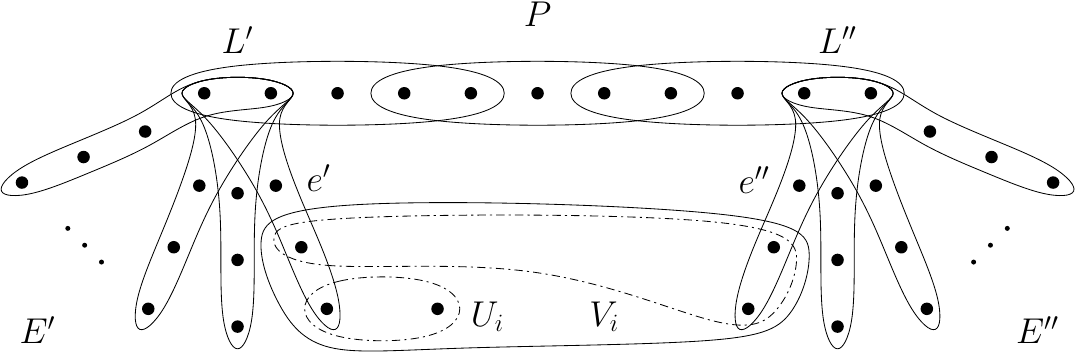}
    \caption{Phase 2 of building $C_6^{5,2}$ for $r=2$.}
    \label{fig:cycle_phase2}
\end{figure}

For ease of calculations, we will bound from below the probability that $U_i$ has this desired property  by adding the constraint that $$U_i\cap\bigcup_{e\in E'\setminus\{e'\}}(e\setminus L')=\emptyset.$$ Setting $z=|V(P)|$ for $m\ge4$ and $z=\ell$ for $m=3$, the probability that $U_i$  satisfies the stronger property is
$$\frac{t'(s-\ell)\binom{n-z-t'(s-\ell))}{r-1}}{\binom nr}=\Theta\left(t/n\right)$$
and so, the probability that it will not happen at all during Phase 2 is, by the chain formula,
$$\left(1-\Theta(t/n)\right)^{t_2}\le \exp\{-\Theta(t^2/n)\}=\exp\{-\Theta(\omega^2)\}=o(1).$$
Hence, a.a.s.\ it will happen at least once during Phase 2 and the player will be able to construct a copy of $H$.


\subsubsection*{Case $H=C_m^{(s,\ell)}$, $s-r\le2\ell-2$.} For $s-r=2\ell-2$ we could basically repeat the above argument. However, for smaller values of $s-r$, due to the threshold being of order $\Omega(n)$, it stops working. The reason is that we cannot have more than $n$ disjoint sets. Therefore, we unify our approach and present a proof valid for all cases when $s-r=2\ell-x$, $2\le x\le \min\{r,\ell\}$ --- the upper bound on $x$ follows from the assumptions $s\ge 2\ell$ and $s\ge r+\ell$. Recall that we have  $\tau^{(r)}_{H}\ge n^{x/3}$ and are after a matching upper bound. Let $t=\omega n^{x/3}$, where $\omega:=\omega(n)\to\infty$ arbitrarily slowly.

Now, we are ready to present player's strategy which a.a.s.\ results in creating a copy of $H$ in $G_t^{(s,r)}$.
Phase 0 is the same as in the case $x=1$. Before Phase~1, in addition to fixing sets $L'$ and $L''$ ($L$ for $m=3$), we partition the vertex set $[n]\setminus V(P)$ ($[n]\setminus L$ for $m=3$) into three sets $W_1,W_2,W_3$ of sizes $n_j=|W_j|\sim n/3$, $j=1,2,3$ (in fact, $n_j=\Theta(n)$ would suffice).

In Phase~1, which lasts $t_1:=\lfloor(t-m+3)/2\rfloor$ steps, every time a set $U_i$ is contained in $W_1$ ($W_2$, resp.) the player extends it by including in $V_i$ the set $L'$ ($L''$, resp.) plus some arbitrary $\ell-x$ vertices of $W_1$ ($W_2$, resp.) (see Figure~\ref{fig:cycle2_phase1}). Clearly, there is a constant $c>0$ such that a.a.s.\ at least $ct$ sets $U_i$ are contained in $W_1$ and the same holds for $W_2$. The two sets of edges obtained that way are denoted $E'$ and $E''$, resp.

\begin{figure}[h]
    \centering
    \includegraphics[scale=0.7]{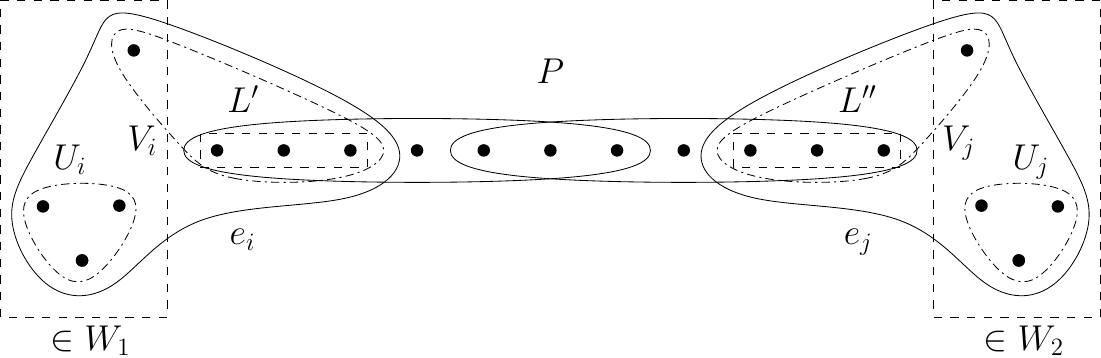}
    \caption{Phase 1 of building $C_5^{7,3}$ for $r=3$.}
    \label{fig:cycle2_phase1}
\end{figure}

In Phase~2, which lasts $t_2=t_1$ steps, the player waits until an edge $U_i$ arrives which, for some $e'\in E'$ and $e''\in E''$, contains exactly $\lfloor x/2\rfloor$ vertices from $e'\setminus L'$, $\lceil x/2\rceil$ vertices from $e'' \setminus L''$, and all remaining $r-x$ vertices belong to $W_3$.
Then the player simply extends $U_i$ by adding $\ell-\lfloor x/2\rfloor$ vertices of $e'\setminus L'$ and  $\ell-\lceil x/2\rceil$ vertices of $e''\setminus L''$. Then $P$ plus the edges $e',e_i,e''$ form a copy of $C_m^{(s,\ell)}$ (Figure~\ref{fig:cycle2_phase2}).

\begin{figure}[h]
    \centering
    \includegraphics[scale=0.7]{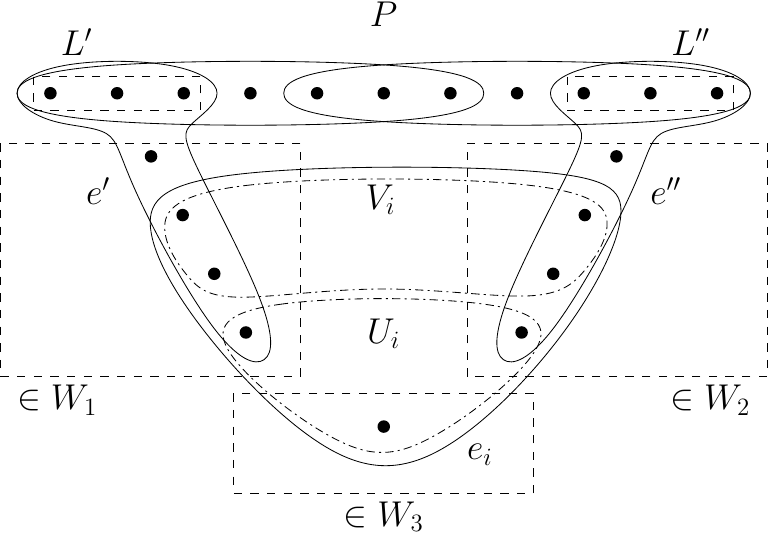}
    \caption{Phase 2 of building $C_5^{7,3}$ for $r=3$.}
    \label{fig:cycle2_phase2}
\end{figure}

By Lemma~\ref{trice}, there are at least $ct\binom r{\lfloor x/2\rfloor}/3=c_1t$  \emph{distinct} $\lceil x/2\rceil$-element sets contained in sets $e'\setminus L'$, $e'\in E'$, and the same is true for
$\lceil x/2\rceil$-element sets in $e''\setminus L''$, $e''\in E''$.
Hence, the probability that $U_i$ has the desired property is at least
$$\frac{(c_1t)^2\times\binom{n_3}{r-x}}{\binom nr}=\Theta(t^2/n^x).$$
Consequently, the probability that it will not happen at all during Phase 2 is, by the chain formula,
   $$\left(1-\Theta(t^2/n^x)\right)^{t_2}\le \exp\{-\Theta(t^3/n^x)\}=\exp\{-\Theta(\omega^3)\}=o(1).$$
This completes the proof of Proposition~\ref{prop:rsl}. \end{proof}

\section{Open Questions}

An interesting question is for what $H$ the weak lower bound from Theorem~\ref{thm:lower_bound_general}, or more generally from Corollary \ref{cor:lower_bound}, yields the correct value of $\tau^{(r)}(H)$. In Theorem~\ref{thm:upper_bound_balanced} we described a broad class of such hypergraphs, but we doubt it is complete.

\begin{problem}
Given $1<r<s$, determine \emph{all} $s$-graphs $H$ for which $\tau^{(r)}(H)=n^{r- \frac1{\mu_G^{(r,s)}}}$.
\end{problem}

A more ambitious goal is to pinpoint the threshold $\tau^{(r)}(H)$ in full generality.
\begin{problem}\label{P2}
Given $1<r<s$, determine $\tau^{(r)}(H)$ for all $s$-graphs $H$.
\end{problem}

So far, beyond Theorem~\ref{thm:upper_bound_balanced}, we succeeded only for $\ell$-tight $s$-uniform paths  and cycles  under, however, quite strong assumptions on $\ell$ and $r$ (see Proposition~\ref{prop:rsl}). Thus,  a first modest task could be to solve Problem~\ref{P2} for the remaining cases of paths and  cycles. Another target class is that of complete $s$-graphs where, except for a handful of small cases, we only have some lower and upper bounds.

\bibliographystyle{abbrv}
\bibliography{refs}

\section*{Appendix}

\subsection*{Edge-balanced implies balanced}

Recall that  an $r$-graph $F$, $r\ge2$, is edge-balanced if for all sub-$r$-graphs  $F'\subset F$ with $e_{F'}\ge1$ we have $g(F')\le g(F)$, where $g(F)=1/r$ if $e_F=1$ and  $g(F)=\frac{e_F-1}{v_F-r}$ if $e_F>1$. Further, we call $F$ balanced if $e_{F'}/v_{F'}\le e_{F}/v_{F}$ for all sub-$r$-graphs $F'$ of $F$.

\begin{claim}\label{ebal->bal}
If an $r$-graph $F$ is edge-balanced, then it is also balanced.
\end{claim}
\proof Let $F$ be an  edge-balanced $r$-graph. Observe that $\delta(F)\ge1$, since otherwise it would not be edge-balanced. Set $v:=v_F$ and $e:=e_F$. Let $F'\subset F$, $e':=e_{F'}\ge1$. Setting also $v':=v_{F'}$, we have by assumption that
\begin{equation}\label{eb}
\frac{e'-1}{v'-r}\le\frac{e-1}{v-r}.
\end{equation}
Our goal is to show that
\begin{equation}\label{b}
\frac{e'}{v'}\le\frac{e}{v}.
\end{equation}
Equation \eqref{eb} is equivalent to
\begin{equation}\label{eb'}
ve'+(v+re)\le v'e+(v+re')
\end{equation}
as well as to
\begin{equation}\label{eb''}
\frac{e-e'}{v-v'}\ge\frac{e-1}{v-r}.
\end{equation}
It follows from \eqref{eb'} that if
\begin{equation}\label{3}
v'+rv\ge v+re',\quad\mbox{or, equivalently,}\quad \frac{e-e'}{v-v'}\ge\frac1r,
\end{equation}
then \eqref{b} must hold. Now, \eqref{3} follows from \eqref{eb''} and
\begin{equation}\label{4}
\frac{e-1}{v-r}\ge\frac1r.
\end{equation}
Finally, \eqref{4} is equivalent to $v\le re$ which is trivially true, as $F$ has no isolated vertices. \qed

\subsection*{Balanced starpluses}

Recall that for $2\le r<s$, an $s$-graph $H$ is \emph{$r$-balanced} if for every subgraph $H'\subset H$ with at least one edge,
$f^{(r)}(H')\le f^{(r)}(H)$, where $f^{(r)}(H)=\frac {|E(H)|}{|V(H)|-s+r}$.
By just comparing the statements of Theorem~\ref{thm:upper_bound_balanced} and Corollary~\ref{cor:lower_bound} it follows that all starpluses satisfying the assumptions of the former statement must be  $r$-balanced. Nevertheless, we provide here a direct proof of this fact in the special case of full $(s,s-r)$-starpluses, which can be viewed as a double check of the correctness of our results.

\begin{proposition}\label{starbal} Let $2\le r<s$ and let $H$ be  a full $(s,s-r)$-starplus on $k$ vertices with excess $\lambda$ satisfying inequality \eqref{eqn:ell}. Then $H$ is $r$-balanced.
\end{proposition}
\proof Let $H'\subset H$, $H'\neq H$. Without loss of generality, we assume that $H'$ is an induced sub-$s$-graph of~$H$. Let $s\le k'<k$ be the number of vertices of $H$ and set $c=s-r$. If $k'=s$, then $e(H')=1$ and so $f^{(r)}(H')=1/r$, while
$$f^{(r)}(H)=\frac{\binom{k-c}r+\lambda}{k-c}\ge\frac{\binom{k-c}r}{k-c}\ge\frac1r ,$$
because the last inequality is equivalent to $\binom{k-c-1}{r-1}\ge1$.

Assume from now on that $k'\ge s+1$. We do not know how many vertices of the center of $H$ belong to $H'$, but nevertheless, the number of edges of $H'$ can be bounded from above by $\binom{k'-c}r+\lambda$. We are going to prove that
\begin{equation}\label{strict}
\frac{\binom{k-c}r+\lambda}{k-c}>\frac{\binom{k'-c}r+\lambda}{k'-c},
\end{equation}
which is, in fact, a bit stronger statement than what is claimed. First note that \eqref{strict} is equivalent to
$$\frac{\binom{k-c}r}{k-c}>\frac{\binom{k'-c}r}{k'-c}+\frac{k-k'}{(k-c)(k'-c)}\lambda.$$ As, by \eqref{eqn:ell}, $\lambda<\tfrac{r\binom{k-c}r}{k-s}$, the above inequality, and thus, \eqref{strict} itself, follows from
$$
\frac{\binom{k-c}r}{k-c}\ge\frac{\binom{k'-c}r}{k'-c}+\frac{(k-k')r\binom{k-c}r}{(k-c)(k'-c)(k-s)}
$$
which, in turn, is equivalent to
\begin{equation}\label{foll}
(k'-c)(k-c)_r\ge(k-c)(k'-c)_r.
\end{equation}
To prove \eqref{foll}, we consider three cases with respect of $k-k'$. Assume first that $k-k'\ge3$ and transform \eqref{foll} to
$$\left(1+\frac{k-k'}{k-c}\right)\cdots\left(1+\frac{k-k'}{k-s+2}\right)\left(1+\frac1{k-s}\right)\ge1+\frac1{k'-s}.$$
Imagining the  left-hand-side completely cross-multiplied, we infer that the above inequality follows from
$$\sum_{i=c}^{s-2}\frac{k-k'}{k-i}+\frac1{k-s}\ge\frac1{k'-s}.$$
As $c=s-r\ge s-2$, the sum above has at least one summand and the L-H-S can be bounded from below by $\tfrac{k-k'}{k'-s+2}$ which, in turn, is at least $\tfrac1{k'-s}$.

When $k-k'=1$, setting $x:=k-s$, \eqref{foll} becomes $(x-1)(x+r)\ge x^2$, equivalently, $x\ge\tfrac r{r-1}$ which is true, because $k\ge k'+1\ge s+2$.
Similarly, when $k-k'=2$, \eqref{foll} becomes $(x-2)(x+r)(x+r-1)\ge x^2(x-1)$. As $r\ge2$, the latter follows from $(x-2)(x+2)(x+1)\ge x^2(x-1)$ which is true for $x\ge3$. But $k\ge k'+2\ge s+3$, so indeed $x=k-s\ge3$. \qed

\subsection*{Properties of $\ell$-tight paths and cycles}
In this subsection we prove some properties of $\ell$-tight paths and cycles used in the main body of the paper. We begin with an observation which follows from the definitions of both structures.

\begin{obs}\label{ob} The following two statements are true:
\begin{itemize}
\item[(i)] Every induced sub-$s$-graph of $P_m^{(s,\ell)}$ is a spanning sub-$s$-graph of a path $P_{m'}^{(s,\ell)}$ for some $m'\le m$.
\item[(ii)] Every induced and \emph{proper} sub-$s$-graph of $C_m^{(s,\ell)}$ is a spanning sub-$s$-graph of a path $P_{m'}^{(s,\ell)}$ for some $m'<m$. \qed
\end{itemize}
\end{obs}

Recall that $d(H)=\max\{\delta(H'):\; H'\subseteq H\}$ is the degeneracy of a hypergraph $H$.
\begin{claim}\label{degen} For all $1\le\ell<s$ and $m\ge1$ we have
 $d(P_m^{(s,\ell)})=1$, while for $m\ge \lfloor (s+1)/(s-\ell)\rfloor$, we have $d(C_m^{(s,\ell)})=\lfloor\frac s{s-\ell}\rfloor$.
 \end{claim}

\proof For the first statement observe that $P_m^{(s,\ell)}$ as well as every sub-$s$-graph of $P_m^{(s,\ell)}$ contains a vertex of degree 1 (in fact, there are at least two such vertices). As for the cycle, by Observation~\ref{ob}(ii) and the first part of this proof, it suffices to consider only $H'=C_m^{(s,\ell}$. Then, the conclusion follows, because $\delta(C_m^{(s,\ell})=\lfloor\frac s{s-\ell}\rfloor$.
\qed

\medskip

Recall that for $2\le r\le s$ and an $s$-graph $H$ with at least $s$ vertices
	$$f^{(r)}(H)=\frac {|E(H)|}{|V(H)|-s+r}\quad\mbox{and}\quad\mu^{(r)}(H)=\max_{H' \subseteq H,\ |V(H')|\ge s}f^{(r)}(H').$$

\begin{claim}\label{miu} For all $r\ge2$, $1\le\ell<s$,
$$\mu^{(r)}(P_m^{(s,\ell)})=\begin{cases}\frac1r \quad\mbox{for}\quad r\le s-\ell\\\frac{m}{(s-\ell)m+\ell-s+r} \quad\mbox{otherwise}
\end{cases}
$$
and, assuming $m\ge \lfloor (s+1)/(s-\ell)\rfloor$,
$$\mu^{(r)}(C_m^{(s,\ell)})=\max\left\{\frac m{(s-\ell)m-s+r},\frac 1r \right\}=\begin{cases}
\frac1r \quad\mbox{for}\quad r\le s-2\ell \\\frac m{(s-\ell)m-s+r}\quad\mbox{for}\quad r\ge s-\ell.
\end{cases}
$$
In particular, for $r\ge s-\ell$, $C_m^{(s,\ell)}$ is $r$-balanced.
 \end{claim}

\proof By Observation~\ref{ob} it suffices to consider only those (proper) sub-$s$-graphs of $P_m^{(s,\ell)}$ and $C_m^{(s,\ell)}$ which are $\ell$-tight paths themselves. Thus,
$$\mu^{(r)}(P_m^{(s,\ell)})=\max_{1\le m'\le m}\frac{m'}{(s-\ell)m'+\ell-s+r}$$
and
$$\mu^{(r)}(C_m^{(s,\ell)})=\max\left\{\frac m{(s-\ell)m-s+r},\;\max_{1\le m'<m}\frac{m'}{(s-\ell)m'+\ell-s+r}\right\}.$$

As function $f(x)=\frac x{(s-\ell)x+\ell-s+r}$ has the derivative $f'(x)=\frac{\ell+r-s}{((s-\ell)+\ell-s+r)^2}$, we have
$f(x)\le f(1)= 1/r$ whenever $\ell-s+r\le0$  whereas
$$f(x)\le f(m)=\frac{m}{(s-\ell)m+\ell-s+r}<\frac m{(s-\ell)m-s+r}$$ whenever $\ell-s+r\le0$. This yields the formula for $\mu^{(r)}(P_m^{(s,\ell)})$ and the  left-hand-side formula for $\mu^{(r)}(C_m^{(s,\ell)})$. Finally, notice that for $r\le s-2\ell$,
$$m\ge 2\ge \frac{s-r}{s-r-\ell}$$
which implies that
$$\frac m{(s-\ell)m-s+r}\le\frac1r.$$
\qed

\medskip

Recall that for an $s$-graph $F$, $s\ge2$, its \textbf{density}  $g(F)$ is defined as $1/s$ if $e_F=1$ and  $\frac{e_F-1}{v_F-s}$ if $e_F>1$, and that we call $F$ \textbf{edge-balanced} if for all sub-hypergraphs  $F'\subset F$ with $e_{F'}>0$ the inequality $g(F')\le g(F)$ holds.

\begin{claim}\label{edge-bal}
For all $s\ge2$ and $m\ge s+1$, the tight cycle $C_m^{(s)}$ is edge-balanced.
 \end{claim}
 \proof  Recall that $P_m^{(s)}$ has exactly $m+s-1$ vertices.
We have $g(C_m^{(s)})=\frac{m-1}{m-s}>1$. Moreover, for every induced proper subgraph $F$ of $C_m^{(s)}$ we have, for some $m'<m$,  $g(F)\le g(P_{m'}^{(s)})=\frac{m'-1}{(m'-s+1)-s}=1$ which finishes the proof. \qed

\medskip
In particular, for $m=s+1$, we infer that the clique $K_{s+1}^{(s)}$ is edge-balanced.
Below,  we show that all hyper-cliques are edge-balanced. (We switch from $s$- to $r$-uniformity, as we apply this result to $H_1$ -- see Section~\ref{match}.)
\begin{claim}\label{clicbal}
For all $2\le r<t$, the $r$-uniform clique $K_t^{(r)}$ on $t$ vertices is edge-balanced.
\end{claim}
\proof
We have to show that for every $t\ge q\ge r+2$,
$$\frac{\binom qr-1}{q-r}  \ge \frac{\binom{q-1}{r}-1}{q-1-r}$$
which is equivalent to
$$(q-1-r)\binom qr+1\ge(q-r)\binom{q-1}r.$$
Skipping $+1$ we get a stronger inequality, equivalent, after cancelation, to $q(q-1-r)\ge(q-r)^2$. This one, in turn, is valid, since for $r\ge2$ we have $q\ge r+2\ge\frac{r^2}{r-1}$.
\qed

\subsection*{Proof of inequality \eqref{in1}}

\begin{claim}\label{ineq} If  $s \ge 3$ and $s < k\le 2s-1$, then
$$\binom{k-1}{s} \le \frac{(s-1) \binom{k-1}{s-1}-(k-1)}{k - s}.$$
\end{claim}
\begin{proof}
First, note that if $k=s+1$ the equation becomes
$$\binom{s}{s} \le (s-1) \binom{s}{s-1}-(s)$$
which is easily seen to be true.

Thus we may suppose that $k \ge s+2$.
We have
\begin{alignat*}{2}
    & &  k-s &\le s-1   \\
   &\implies \quad &   (k-s)^2 &\le (s-1)^2 \\
   &\implies &  (k-s)^2 &\le  s(s-1) - (s-1) \\
   &\implies &  \frac{k-s}{s} &\le \frac{s-1}{k-s} -
      \frac{s-1}{s(k-s)}  \\
   &\implies &
       \frac{k-s}{s}\binom{k-1}{s-1} &\le
     \frac{s-1}{k-s}\binom{k-1}{s-1} - \frac{s-1}{s(k-s)}\binom{k-1}{s-1} \\
   &\implies &
        \binom{k-1}{s} &\le
        \frac{s-1}{k-s}\binom{k-1}{s-1} -
       \frac{s-1}{s(k-s)}\binom{k-1}{s-1}
\end{alignat*}
so it suffices to show that
\begin{equation*}
    \frac{s-1}{s} \binom{k-1}{s-1} \ge k-1,
\end{equation*}
or equivalently that
\begin{equation}\label{eqn:binomials}
    \frac{s-1}{k-1} \binom{k-1}{s-1} =  \binom{k-2}{s-2}  \ge s.
\end{equation}
Since $\binom{k-2}{s-2} \ge \binom{s}{s-2} = \frac{s(s-1)}{2}$ and $s \ge 3$, equation \eqref{eqn:binomials} holds and the proof is complete.

\end{proof}

\subsection*{Functions $f_s(k,\ell)$ and $\ell_k(r,s)$}

Recall that function  $f_s(k,\ell)=\frac{(k)_s-(\ell)_s}{k-\ell}$ has appeared  in the exponent of the upper bound on the threshold $\tau^{(r)}(K_k^{(s)})$ in Theorem~\ref{thm:general_cliques}, while $\ell:=\ell_k(r,s)$ was the smallest integer $\ell$ such that
\begin{equation}\label{constraint_rep}
k-\ell-r-\frac{k-\ell}{\binom ks-\binom{\ell}{s}}\sum_{j=1}^r\binom\ell{s-j}\left[\binom{k-\ell}j-\binom rj\right]\le0.
\end{equation} (Here, for convenience, we repeat inequality~\eqref{constraint} from Section~\ref{2.3}.)

We first show that
\begin{equation}\label{mono}
f_s(k,\ell)\quad\mbox{ is (strictly) increasing in both variables}.
\end{equation} Indeed, one can show the recurrence
$f_s(k,\ell)=\ell f_{s-1}(k-1,\ell-1)+(k-1)_{s-1}$, which implies, by induction on $s$, that
$$f_s(k,\ell)=\sum_{i=1}^s(\ell)_{s-i}(k-s+i-1)_{i-1}.$$ This form reveals that $f_s(k,\ell)$ is increasing in $\ell$, as well as in $k$.
In particular, it follows that with $\ell:=\ell_k(r,s)$ and $\bar\ell$ being the smallest integer $\bar\ell$ satisfying~\eqref{constraint_rep} with $k$ and $\ell$ replaced, respectively, by $\ell$ and $\bar\ell$, we have $f_s(k,\ell)>f_s(\ell,\bar\ell)$ (since $\ell\le k-r<k$ and $\bar\ell\le\ell-r<\ell$).

 Next, we are going to determine $\ell_k:=\ell_k(2,3)$ explicitly, that is, to prove~\eqref{elka}. For $r=2$ and $s=3$,~\eqref{constraint_rep} becomes
$$k-\ell-2-\frac{6(k-\ell)\ell}{(k)_3-(\ell)_3}\left[\binom{k-\ell}2-1+\frac{\ell-1}2(k-\ell-2)\right]\le0,$$
which, in turn, is equivalent to
\begin{equation}\label{equiv}
\ell^2-(2k+3)\ell+k^2-3k+2\le0.
\end{equation}
 By solving the above quadratic inequality, we obtain
$$\ell_k=\left\lceil k+\frac32-\sqrt{6k+1/4}\right\rceil,$$
the same formula which appears in~\eqref{elka} and in Proposition~\ref{Kk}.

It has been mentioned earlier that $s-r\le \ell_k(r,s)\le k-r$. However, for $\ell_k(2,3)$ the upper bound can be sharpened under a mild assumption on $k$.
 Indeed, dropping the ceiling,
  \begin{equation}\label{k-2}
  \ell_k\le k+\frac52-\sqrt{6k+1/4}\le k-a
  \end{equation}
  for all $k\ge a(a+5)/6+1$. E.g., $\ell_k\le k-3$ for $k\ge5$, while $\ell_k\le k-4$ for $k\ge7$.

It is not easy to compute $\ell_k(r,s)$ in general. We have made an attempt at the next smallest case: $r=2$, $s=4$.
In this case \eqref{constraint_rep} becomes
$$k-\ell-2\frac{4!}{f_4(k,\ell)}\left\{\binom{\ell}3(k-\ell-2)+\binom{\ell}2\left[\binom{k-\ell}2-1\right]\right\}\le0,$$
equivalently
$$k^4+8k\ell^3+20k\ell^2+26k\ell+23k^2+12\le 3\ell^4+8k^2\ell+6k^2\ell^2+8k^3+4\ell^3+5\ell^2+20\ell,$$
which, after setting $x=k-\ell$ becomes
$$4kx^3+8kx^2+21k^2+10kx+20x+12\le 3x^4+20k^2x+5x^2+20k.$$
Assuming  $k$ is large and focusing on the leading terms, $4kx^3$ on the left and $3x^4+20k^2x$ on the right, it is easy to show that $x=O(\sqrt k)$, and so $\ell_k(2,4)=k-\Omega(\sqrt k)$. Indeed, first note that $x=O(k^{2/3})$, since otherwise $4kx^3-3x^4>x^4\gg k^2x$, a contradiction. So, assume that $x=\omega(k)\sqrt k$, where $\omega(k)\to\infty$, but $\omega(k)=O(k^{1/6})$. Now the left-hand-side is $\Theta(\omega(k)^3k^{5/2})$, while the right-hand-side is $\Theta(\omega(k)^4k^2+\omega(k)k^{5/2})$, a contradiction again.

\end{document}